\documentclass[a4paper,11pt]{newLag}
\makeatletter

\usepackage[T1]{fontenc}
\usepackage{amsmath}
\usepackage{mathtools}
\usepackage{amsthm}
\usepackage{amsfonts}  
\usepackage{amssymb}
\DeclareOldFontCommand{\rm}{\normalfont\rmfamily}{\mathrm}
\DeclareOldFontCommand{\sf}{\normalfont\sffamily}{\mathsf}
\DeclareOldFontCommand{\tt}{\normalfont\ttfamily}{\mathtt}
\DeclareOldFontCommand{\bf}{\normalfont\bfhttps://www.overleaf.com/project/62022d046a9b74af4550fc2aseries}{\mathbf}
\DeclareOldFontCommand{\it}{\normalfont\itshape}{\mathit}
\DeclareOldFontCommand{\sl}{\normalfont\slshape}{\@nomath\sl}
\DeclareOldFontCommand{\sc}{\normalfont\scshape}{\@nomath\sc}
\makeatother
\usepackage{leftidx}
\usepackage[utf8]{inputenc}
\usepackage{bm}
\usepackage{graphicx}
\usepackage{graphicx}
\usepackage{booktabs}
\usepackage{hyperref}
\usepackage{cancel}

\usepackage{tikz} 
\tikzstyle{every pin}=[
					   pin edge = black,					   
					   font=\footnotesize]
\tikzstyle{block} = [draw, fill=blue!20, rectangle, minimum height=3em, minimum width=6em]
\tikzstyle{sum} = [draw, fill=blue!20, circle, node distance=1cm]
\tikzstyle{input} = [coordinate]
\tikzstyle{output} = [coordinate]
\tikzstyle{pinstyle} = [pin edge={to-,thin,black}]

\usepackage{etoolbox} 

\ltdsetup{\today}{}{}

\newcommand{\newL}{\tilde{\mathcal{L}}}
\newcommand{\newJ}{\tilde{\mathcal{J}}}

\usepackage{caption} 
\usepackage{subcaption} 
\usepackage[font=small]{caption}
\usepackage[utf8]{inputenc}
\usepackage{enumitem}
\usepackage[backend=biber,
			style=numeric,
                sorting=nyt,
                sortcites=true,
			isbn=false,
			url=false,
			eprint=false,
                maxbibnames=5,
                giveninits=true,
                autopunct=true,
                autolang=hyphen,
                hyperref=true,
                abbreviate=true,
                backref=true,
                citestyle=ieee,backref=false]{biblatex}

\DeclareNameAlias{default}{last-first}

\newtheorem{theorem}{Theorem}[section]
\newtheorem{definition}[theorem]{Definition}

\newtheorem{proposition}[theorem]{Proposition}

\newtheorem{remark}[theorem]{Remark}
\newtheorem{corollary}[theorem]{Corollary}

\title{Variational integrators for a new Lagrangian approach to 
control affine systems with a quadratic Lagrange term}
\date{\today}

\author
{
Michael Konopik\footnote{\textit{First author}, \textit{corresponding author}. Friedrich-Alexander-Universität Erlangen-Nürnberg (FAU), Institute of Applied Dynamics (LTD), Immerwahrstrasse 1, 91058 Erlangen, Germany. Email: \href{mailto:michael.konopik@fau.de}{michael.konopik@fau.de}}\ \thanks{The work of this author has been supported by Deutsche Forschungsgemeinschaft (DFG), Grant No. LE 1841/12-1, AOBJ: 692092.}\ \qquad
Sigrid Leyendecker\footnote{Friedrich-Alexander-Universität Erlangen-Nürnberg (FAU), Institute of Applied Dynamics (LTD), Immerwahrstrasse 1, 91058 Erlangen, Germany. Email: \href{mailto:sigrid.leyendecker@fau.de}{sigrid.leyendecker@fau.de}}\ \qquad
 Sofya Maslovskaya\footnote{Universität Paderborn (UPB), Numerical Mathematics and Control (NMC), Warburger Straße 100, 33098 Paderborn, Germany. Email: \href{mailto:sofya.maslovskaya@upb.de}{sofya.maslovskaya@upb.de}}\ \thanks{The work of this author has been supported by Deutsche Forschungsgemeinschaft (DFG), Grant No. OB 368/5-1, AOBJ: 692093}
\\
Sina Ober-Bl\"obaum\footnote{  Universität Paderborn (UPB), Numerical Mathematics and Control (NMC), Warburger Straße 100, 33098 Paderborn, Germany. Email: \href{mailto:sinaober@math.uni-paderborn.de}{sinaober@math.uni-paderborn.de}}\
\ \qquad
Rodrigo T.~Sato Mart{\'\i}n de Almagro\footnote{\textit{First author}. Friedrich-Alexander-Universität Erlangen-Nürnberg (FAU), Institute of Applied Dynamics (LTD), Immerwahrstrasse 1, 91058 Erlangen, Germany. Email: \href{mailto:rodrigo.t.sato@fau.de}{rodrigo.t.sato@fau.de}}\
}

\begin{document}
\maketitle
\section*{Abstract}
 In this work, we analyse the discretisation of a recently proposed new Lagrangian approach to optimal control problems of affine-controlled second-order differential equations with cost functions quadratic in the controls. We propose exact discrete and semi-discrete versions of the problem, providing new tools to develop numerical methods. Discrete necessary conditions for optimality are derived and their equivalence with the continuous version is proven. A family of low-order integration schemes is devised to find approximate optimality conditions, and used to solve a low-thrust orbital transfer problem. Non-trivial equivalent standard direct methods are constructed. Noether's theorem for the new Lagrangian approach is investigated in the exact and approximate cases.

    \vspace{2mm}

    \textbf{\bfseries \sffamily Keywords:} \textit{optimal control}, \textit{mechanical systems}, \textit{variational methods}, \textit{necessary conditions}, \textit{discretization}, \textit{discrete mechanics}, \textit{symplectic methods}, \textit{numerical integration}, \textit{symmetries}.

    \vspace{2mm}
    
\textbf{\bfseries \sffamily Mathematics Subject Classification:} 37M15,
49K15,
49M05,
49M25,
65K10,
65P10,
70H15,
70G45,
70G65,
70Q05.

\section{Introduction}

Optimal control of mechanical systems is of great importance both in engineering applications, e.g.~in robotics, and from a theoretical point of view. Such problems naturally possess a rich internal structure which can have important consequences for the overall behaviour of the system.\\
Indeed, on the one hand, the configuration of a forced mechanical system evolves on a manifold $\mathcal{Q}$, and its evolution can be described by the Lagrange-D'Alembert principle, leading to second-order differential equations with flows on either the velocity or momentum-phase space, $T\mathcal{Q}$ or $T^*\mathcal{Q}$ respectively. The latter space comes naturally equipped with a symplectic structure which, under some regularity assumptions, can be carried over to the former. When subjected solely to conservative forces, these flows can be shown to be symplectic, i.e.~they preserve this symplectic structure. Besides, by the celebrated Noether's theorem, we know that the symmetries of a mechanical system lead to conserved quantities \cite{ArnoldMethods,AbrahamMarsden}.\\
On the other hand, the necessary conditions for optimality for an optimal control problem (OCP) of a dynamical system evolving on a state space $\mathcal{M}$ are provided by Pontryagin's maximum / minimum principle (PMP) \cite{pontryagin1964}, whose overall structure is Hamiltonian. In the PMP, the state of the dynamical system, its accompanying costate or adjoint state and the control variables evolve on a presymplectic manifold. Moreover, at the optimum, the resulting flow in state-adjoint space, $T^*\mathcal{M}$, is symplectic. Also the Hamiltonian version of Noether's theorem can be applied in this setting \cite{SussmannSymmetries93,Blankenstein01,EcheverriaEnriquez03}. In the case of control of mechanical systems the state space $\mathcal{M}$ can be either velocity or momentum-phase space, leading to an interesting nested structure for the total problem.\\
Note that the PMP can be regarded as a non-smooth generalisation of variational calculus, and thus, under certain regularity assumptions, OCPs can be treated with the latter. This makes the parallels between optimal control and Lagrangian mechanics even more obvious and it can also expose further geometric facets of the problem.\\
%

In a recent article 
\cite{leyendecker2024new} a new Lagrangian approach to optimal control of second-order systems was proposed. The authors exploit this variational structure, and by performing integration by parts before the process of variation takes place, they were able to introduce a new hyperregular Lagrangian with state positions and velocity adjoints as configuration variables. Several of the optimality conditions of the OCP are transformed into Euler-Lagrange equations. They also explored an associated new Hamiltonian formulation and some consequences of the symmetries of the new Lagrangian.

Since analytical solutions for most of these problems are generally not available, we need to rely on numerical methods to approximate them. However, numerical methods do not generally respect the previously outlined structure. This is precisely where geometric or structure-preserving integration comes into the picture.\\
Given the variational structure of the OCP and the regularity of the new Lagrangian, variational integrators \cite{marsden2001discrete} become a natural choice to construct such structure preserving methods. These discretise the variational principle itself, are known to be symplectic and are built over solid theoretical grounds. In particular, they are based on approximations of the so-called exact discrete Lagrangian, which bridges the realms of discrete mechanics on $\mathcal{Q} \times \mathcal{Q}$ and continuous mechanics on $T\mathcal{Q}$.\\


The aim of this work is to provide a discrete counterpart to \cite{leyendecker2024new}, exploring and analysing the discretisation of this formulation in the vein of discrete mechanics. For this, we define corresponding exact discrete control-independent and control-dependent Lagrangians. These let us analyse the discrete new Lagrangian formulation with generality and allow us to approximate them with warranties. Moreover, we expect these ideas to become important tools in the future.

Typical approaches to solving optimal control problems numerically can be usually split into two categories \cite{formalskii10,betts2010}. The first class corresponds to \textit{optimise-then-discretise} or \textit{indirect discretisation} approaches where the minimisation is first carried out using the PMP or variational calculus. Then, the resulting necessary optimality conditions are discretised and solved numerically \cite{pontryagin1964,BrysonHo75}. 
The second category is the so-called \textit{discretise-then-optimise} or \textit{direct discretisation} approach where the augmented cost function is first discretised and the resulting discrete cost function is optimised. Variation of this discrete cost leads to Karush-Kuhn-Tucker conditions \cite{betts2010,campos15,betsch2017conservation,bottasso2004optimal}. Discrete mechanics and optimal control (DMOC) \cite{oberbloebaum_diss}, where the dynamics and objective function are discretised using a discrete mechanics approach \cite{marsden2001discrete}, belongs to this category. 
Approaches of both categories are related by symplectic discretisation schemes, that utilise the symplectic structure of state-adjoint space. 
These relationships were analysed e.g. in \cite{hager00,bonnans04,sanz-serna2015,ober2011discrete}.

In the paper we construct both discrete control-dependent and control-independent Lagrangians. Since the latter are obtained by discretising a control Lagrangian where the controls have been eliminated via minimisation, the resulting methods can be regarded as indirect. However, our discrete control-dependent Lagrangians lie squarely in the direct camp. By establishing their equivalence, we find a bridge between both approaches.\\

Starting with Section \ref{exactsection}, exact discrete necessary conditions for optimality 
are derived and analysed 
both when controls are left as parameters and also when the controls have been eliminated by imposing minimisation. After having analysed the exact discrete optimality conditions,  a family of low-order methods is derived in Section \ref{loworderapproxsection} by approximating 
the exact discrete new Lagrangians. 
The resulting discrete optimality conditions are then analysed and shown to be equivalent for the control-dependent and independent case.

After having analysed the properties of the new Lagrangian formulation 
in the discrete setting, we then shift our focus to compare it 
with other, more common discretisation approaches. In Section \ref{comparison_section}, we consider two direct approaches; the first discretises the state equations in the form of second-order differential equations, while the second approach discretises the state equations in form of first-order differential equations. We show how it is possible to construct equivalent methods to our low-order integration schemes, provided the discretisation in these direct approaches is carried out judiciously.

In Section \ref{symmetrysection}, Noether's theorem for the optimal control Lagrangians are proven both in the exact and approximate control-dependent cases, showing that affine symmetry actions are preserved by the chosen family of low-order integration schemes. The preservation of symplecticity in the state-adjoint space of numerical methods derived from the new control Lagrangians is discussed.

Finally, in Section \ref{examplesection}, an example optimal control problem in the form of a low-thrust orbital transfer is solved via the family of low-order integration schemes proposed showing well-behaved solution curves even for first-order methods and a small number of integration steps.

\section{Necessary conditions for optimality in the continuous and discrete setting}
\label{exactsection}
Before we start with the discussion of the optimal control problem in the discrete setting, let us summarise the spaces we work in, first in the continuous setting, and summarize the key results of the continuous treatment \cite{leyendecker2024new}, which will be used for comparison and context later.
\subsection{Continuous setting}

Following \cite{leyendecker2024new}, we consider a particular type of optimal control problem, namely, of affine-controlled second-order ordinary differential equations (SODEs) and running costs quadratic in the controls.\\

Let $(\mathcal{Q},\mathrm{g}_{\mathcal{Q}})$ be a Riemannian manifold, and consider its tangent bundle $\tau_{\mathcal{Q}}: T \mathcal{Q} \to \mathcal{Q}$. Further, let $\pi^{\mathcal{E}}: \mathcal{E} \to \mathcal{Q}$ be an anchored vector bundle with typical fibre $\mathcal{N}$, $\dim \mathcal{Q} \geq \dim \mathcal{N}$, and injective linear anchor $\rho: \mathcal{E} \to T\mathcal{Q}$, $\tau_\mathcal{Q} \circ \rho = \pi^{\mathcal{E}}$, i.e. a vector bundle morphism over the identity. We say $\mathcal{Q}$ is our configuration space, $T\mathcal{Q}$ is our state space, $\mathcal{E}$ is our control bundle and $\mathcal{N}$ our control space. We assume adapted local coordinates $(q,v) \in T\mathcal{Q}$ and $(q,u) \in \mathcal{E}$ so that $(q,v,u) \in T\mathcal{Q} \oplus_{\mathcal{Q}} \mathcal{E}$, the Whitney sum of both bundles. $\mathrm{g} \equiv \mathrm{g}_{\mathcal{E}} = \rho^* \mathrm{g}_{\mathcal{Q}}$ defines a Riemannian metric on our space of controls.\\

The optimal control problem (OCP) under consideration is that of finding curves\footnote{We incur in a tacit abuse of notation by using the same symbol for either points and curves. However, we believe that context makes the distinction quite clear.} $q \in C^3([0,T],\mathcal{Q})$ and $u \in C^1([0,T],\mathcal{N})$ such that
\begin{subequations}
	\label{problem_definition}
	\begin{alignat}{2}
		     \min_{(q,u)} \quad &&      \mathcal{J}[q,u] &= \phi(q(T), \dot{q}(T)) + \int_0^{T} \frac{1}{2} \,u(t)^{\top}\mathrm{g}(q(t))\, u(t)~dt \label{runningCost}\\
		\text{subject to} \quad &&        q(0) &= q^0,\label{initialqcondi}\\
		                        &&  \dot{q}(0) &= \dot{q}^0,\label{initialvcondi}\\
		                        && \ddot{q}(t) &= f(q(t),\dot{q}(t))+\rho(q(t)) \, u(t),\label{SODE}
	\end{alignat}
\end{subequations}
where \eqref{SODE} is an affine-controlled SODE, which may be interpreted as a submanifold of the second-order tangent bundle of $\mathcal{Q}$, $T^{(2)} \mathcal{Q}$ \cite{deLeonRodrigues85}, and $(q^0,\dot{q}^0) \in T\mathcal{Q}$. We assume
$\phi, f, \rho, \mathrm{g}$ are continuously differentiable.

Note that whenever $\dim \mathcal{Q} = \dim \mathcal{N}$, by the properties of the anchor, $\rho$ establishes a vector bundle isomorphism between $T\mathcal{Q}$ and $\mathcal{E}$. We say then, that \eqref{SODE} is fully-actuated. If $\dim \mathcal{Q} < \dim \mathcal{N}$, then \eqref{SODE} is under-actuated.\\

Under our conditions, the OCP posed by \eqref{problem_definition} can be equivalently restated as the unconstrained minimisation of the augmented cost function
\begin{align}
    \hat{\mathcal{J}}^{\mathcal{E}}[y,u,\mu,\nu]& = \phi(q(T),\dot{q}(T)) + \mu^{\top}(q(0) - q^0) + \nu^{\top} (\dot{q}(0) - \dot{q}^0) \nonumber\\
	&+ \int_0^{T} \left[ \frac{1}{2} \, u(t)^{\top}\mathrm{g}(q(t))\, u(t) + \lambda^{\top} \left( \ddot{q} - f(q,\dot{q}) + \rho(q) \,u \right)\right] \,dt, \label{augmentedObjective}
\end{align}
Here, $\lambda$ and $(\mu, \nu)$ play the role of Lagrange multipliers for the SODE and the initial conditions respectively. The first receives the name of costate or adjoint variable and it acts as a covector, i.e. for $q \in \mathcal{Q}$, $\lambda \in T_{q}^{*} \mathcal{Q}$ and thus $y = (q,\lambda) \in T^* \mathcal{Q}$.  We take the opportunity here to introduce the canonical projection $\pi_{\mathcal{Q}}: T^*\mathcal{Q} \to \mathcal{Q}$, $\pi_{\mathcal{Q}}(q,\lambda) = q$. We assume that $\hat{\mathcal{J}}^{\mathcal{E}}$ is now defined over curves $y \in C^3([0,T],T^*\mathcal{Q})$, though it is sufficient for the curve $\lambda$ to be only twice continuously differentiable.\\

Taking variations of the augmented objective $\delta \hat{\mathcal{J}}^{\mathcal{E}}$ and requiring its stationarity then results in the following necessary optimality conditions for the optimal control problem.
\begin{subequations}
\label{eulerLagrangeufull}
	\begin{alignat}{2}
    \text{(adjoint dynamics) } && \ddot \lambda &=  \left(\frac{\partial}{\partial q}f(q,\dot{q}) + \frac{\partial}{\partial q}\rho(q)u\right)^{\top}\lambda - \frac{d}{dt}\left(\left(\frac{\partial}{\partial \dot{q}}f(q,\dot{q})\right)^{\top}\lambda\right) - \frac{1}{2}\left(\frac{\partial}{\partial q}\mathrm{g}(q)(u,u)\right)^{\top}\!, \label{eulerLagrangeu1}\\
    \text{(state dynamics) } && \ddot q &= f(q,\dot{q}) + \rho(q) u, \label{eulerLagrangeu2}\\
    \text{(minimisation) } && \mathrm{g}(q)u &= \rho(q)^{\top} \lambda,\label{eulerLagrangeu3}\\
        \text{(transversality) } && \mu &= - \dot{\lambda}(0) - \frac{\partial}{\partial \dot{q}} f(q(0),\dot{q}(0)),\label{boundaryend1}\\
        && \nu &= \lambda(0)\,. \label{boundaryend2}\\
        && \dot{\lambda}(T) &= \frac{\partial }{\partial q}\phi(q(T),\dot{q}(T))^{\top} + \frac{\partial}{\partial \dot{q}} f(q(T),\dot{q}(T))^{\top} \frac{\partial}{\partial \dot{q}}\phi(q(T),\dot{q}(T))^{\top},      \label{boundaryend3}\\
        && \lambda(T) &= -\frac{\partial}{\partial \dot{q}}\phi(q(T),\dot{q}(T))^{\top}.      \label{boundaryend4}
    \end{alignat}
\end{subequations}

\begin{remark}
The transversality conditions \eqref{boundaryend2} and \eqref{boundaryend4} are those associated with variations $\delta \dot{q}(0)$ and $\delta \dot{q}(T)$, respectively. Once in our discrete setting, such variations will not appear naturally in the formulation. However, it will be possible to recuperate them as will be shown.
\end{remark}

\begin{remark}
\label{rmk:OCP_first_order}
In the literature it is standard to construct a different augmented cost function, by first transforming the controlled SODE \eqref{SODE} into a system of ODEs,
\begin{align*}
    \dot{q}(t) &= v(t),\\
    \dot{v}(t) &= f(q(t),v(t)) + \rho(q(t)) u(t),
\end{align*}
and then using this to augment the cost function \eqref{runningCost},
\begin{align}
	\bar{\mathcal{J}}^{\mathcal{E}}[z,u,\mu,\nu]& = \phi(q(T),\dot{q}(T)) + \mu^{\top}(q(0) - q^0) + \nu^{\top} (v(0) - \dot{q}^0) \nonumber\\
	&+ \int_0^{T} \left[ \frac{1}{2} \, u(t)^{\top}\mathrm{g}(q(t))\, u(t) + \lambda_q^{\top} \left( \dot{q} - v \right) + \lambda_v^{\top} \left( \dot{v} - f(q,\dot{q}) + \rho(q) \,u \right)\right] \,dt.\label{eq:standardAugmentedObjective}
\end{align}
Here $z$ is a curve in $T^*(T\mathcal{Q})$ and local coordinates $(q,v,\lambda_q,\lambda_v)$ are being used to express a point on this space. In \cite{leyendecker2024new}, it was shown that both \eqref{augmentedObjective} and this one lead to the same extremals and that it was possible to identify $\lambda$ with $\lambda_v$.\\
Another thing that will be important later on is the fact that $\lambda_q = - \dot{\lambda} - \frac{\partial}{\partial \dot{q}} f(q,\dot{q})$. This means that \eqref{boundaryend1} and \eqref{boundaryend2} reduce to $\mu = \lambda_q(0)$ and $\nu = \lambda_v(0)$, while \eqref{boundaryend3} and \eqref{boundaryend4} reduce to $- \frac{\partial}{\partial q} \phi(q(T),\dot{q}(T)) = \lambda_q(N)$ and $-\frac{\partial}{\partial \dot{q}} \phi(q(T),\dot{q}(T)) = \lambda_v(N)$.
\end{remark}

\begin{definition}
\label{def:OCP_regularity}
Stationarity with respect to variations of $u$ of the augmented cost function leads to algebraic equations. We say that an OCP for a controlled SODE is algebraically regular if said equation is locally solvable for $u$, i.e. there exists open $\mathcal{V} \subset T\mathcal{Q} \oplus_{\mathcal{Q}} T^*\mathcal{Q} \oplus_{\mathcal{Q}} \mathcal{E}$ on which the conditions for the implicit function theorem are satisfied for each $(q,v,\lambda,u) \in \mathcal{V}$.
\end{definition}

Clearly, \eqref{augmentedObjective} is algebraically regular, since by our assumptions $\mathrm{g}$ is invertible and so \eqref{eulerLagrangeu3} is globally solvable, leading to a linear relation between $u$ and $\lambda$.\\

The approach proposed in \cite{leyendecker2024new} reformulates \eqref{augmentedObjective} by applying partial integration on the running cost in order to eliminate $\ddot q$ in favour of $\dot{\lambda}$. This leads to the following equivalent augmented cost function
\begin{align}
	\tilde{\mathcal{J}}^{\mathcal{E}}[y,u,\mu,\nu]& = \phi(q(T),\dot{q}(T)) + \mu^{\top}(q(0) - q^0) + \nu^{\top} (\dot{q}(0) - \dot{q}^0) + \lambda(T)^{\top}\dot{q}(T) - \lambda(0)^{\top}\dot{q}(0) \nonumber\\
	&+ \int_0^{T} \left[\frac{1}{2} \, u(t)^{\top}\mathrm{g}(q(t))\, u(t) - \dot\lambda(t)^{\top} \dot{q}(t) -\lambda(t)^{\top} (f(q(t),\dot{q}(t)) + \rho(q(t)) u(t))\right] \,dt, \label{newAugmentedObjective}
\end{align}
The new augmented running cost can be interpreted as a $u$-parametric family of Lagrangians for a system on $T^*\mathcal{Q}$. Thus, $\tilde{\mathcal{L}}^{\mathcal{E}}: T T^*\mathcal{Q} \oplus_{\mathcal{Q}} \mathcal{E} \to \mathbb{R}$,
\begin{equation}
\label{eq:newControlDepLag}
\tilde{\mathcal{L}}^{\mathcal{E}}(q,\lambda,v,v_{\lambda},u) =  v_{\lambda}^{\top} v + \lambda^{\top} (f(q,v) + \rho(q) u) - \frac{1}{2} \, u^{\top}\mathrm{g}(q)\, u
\end{equation}
defines a \textit{new control-dependent Lagrangian} for the OCP under study.

\begin{remark}
Note that the definition of the new control-dependent Lagrangian we provide here differs by a minus sign from that originally proposed in \cite{leyendecker2024new}. The reason for this redefinition is mostly for future convenience \cite{Konopik25b}. This will be clarified and expanded on a future publication where we extend and further analyse the new Lagrangian approach \cite{Konopik25b}.
\end{remark}

In \cite{leyendecker2024new} it was shown that the Euler-Lagrange equations for this new control-dependent Lagrangian for the OCP are precisely \eqref{eulerLagrangeu1} and \eqref{eulerLagrangeu2}. Moreover, when regarding $u$ as a parameter, this new Lagrangian is hyperregular, i.e. for each $u$, the fibre derivatives
\begin{subequations}
\label{contLegendre}
	\begin{alignat}{2}
        p_q &= \frac{\partial}{\partial v}\tilde{\mathcal{L}}^{\mathcal{E}}(q,\lambda,v,v_{\lambda},u) &&= v_{\lambda} + \frac{\partial}{\partial v}f(q,v)^{\top} \lambda\,,\\
        p_\lambda &= \frac{\partial}{\partial v_{\lambda}} \tilde{\mathcal{L}}^{\mathcal{E}}(q,\lambda,v,v_{\lambda},u) &&= v\,,
    \end{alignat}
\end{subequations}
establish a diffeomorphism between $T(T^*\mathcal{Q})$ and $T^*(T^*\mathcal{Q})$.

\begin{remark}
\label{rmk:pLambdaEqualVelocity}
It is important to note that the canonical momentum associated to the costate $\lambda$, $p_\lambda$, coincides with the state velocity $v$. This will be used extensively in what follows.
\end{remark}

As previously mentioned, given the algebraic regularity of our OCP, it is always possible to obtain $u(q,v,\lambda)$ via \eqref{eulerLagrangeu3}. Substitution into \eqref{eq:newControlDepLag} leads to the definition of the new Lagrangian for the OCP
\begin{equation}
\tilde{\mathcal{L}}(q,\lambda,v,v_{\lambda}) =  v_{\lambda}^{\top} v + \lambda^{\top} f(q,v) + \frac{1}{2} \, \lambda^{\top} b(q)\, \lambda,
\label{eq:newLagrangianNoU}
\end{equation}
where $b(q) = \rho(q) \, \mathrm{g}(q)^{-1} \, \rho(q)^\top$. The Euler-Lagrange equations associated to this new Lagrangian are
\begin{subequations}
\begin{align}
        &\ddot \lambda = \frac{1}{2}\left(\frac{\partial}{\partial q}b(q)(\lambda,\lambda)\right)^{\top}  + \left( \frac{\partial}{\partial q} f(q,\dot{q})\right)^{\top} \lambda-\frac{d}{dt}\left(\left(\frac{\partial}{\partial \dot{q}}f(q,\dot{q})\right)^{\top}\lambda\right)=0 \label{eulerLagrange_no_u1}\\
        &\ddot q = f(q,\dot{q}) + b(q)\lambda \label{eulerLagrange_no_u2}
    \end{align}
    \label{eulerLagrange_no_u}
\end{subequations}
which coincide precisely with the adjoint and state dynamics in \eqref{eulerLagrangeu1} and \eqref{eulerLagrangeu2} respectively, after substitution of this $u(q,v,\lambda)$. These are the necessary conditions for optimality of
\begin{align}
	\tilde{\mathcal{J}}[y,\mu,\nu]& = \phi(q(T),\dot{q}(T)) + \mu^{\top}(q(0) - q^0) + \nu^{\top} (\dot{q}(0) - \dot{q}^0) + \lambda(T)^{\top}\dot{q}(T) - \lambda(0)^{\top}\dot{q}(0)\nonumber\\
	&- \int_0^{T} \left[\frac{1}{2} \, \lambda(t)^{\top} b(q(t))\, \lambda(t) + \dot\lambda(t)^{\top} \dot{q}(t) + \lambda^{\top} f(q(t),\dot{q}(t))\right] \,dt, \label{newAugmentedObjective_no_u}
\end{align}

\subsection{Exact discrete setting}

In this section we proceed to discretise the previously outlined theory using the framework of discrete mechanics. For a detailed introduction to discrete mechanics and its relation to numerical integration, refer to \cite{marsden2001discrete}. In particular, its Section 1.1.1, provides a complete historical breakdown and further references. For its application to optimal control, refer to \cite{oberbloebaum_diss}.\\

Our approach begins by exploring the exact discrete setting, which links the continuous setting with the discrete one and then exploring approximations in order to obtain structure-preserving numerical integration schemes.\\

Consider the augmented cost function $\tilde{\mathcal{J}}[y]$, \eqref{newAugmentedObjective_no_u}, now written in terms of the new Lagrangian for the OCP defined in \eqref{eq:newLagrangianNoU}:
\begin{equation*}
	\tilde{\mathcal{J}}[y,\mu,\nu] = \phi(q(T),\dot{q}(T)) + \mu^{\top}(q(0) - q^0) + \nu^{\top} (\dot{q}(0) - \dot{q}^0) + \lambda(T)^{\top}\dot{q}(T) - \lambda(0)^{\top}\dot{q}(0) - \int_0^{T} \tilde{\mathcal{L}}(y(t),\dot{y}(t)) \,dt.
\end{equation*}
Notice that the integral part coincides with the Lagrangian action associated with $\tilde{\mathcal{L}}$. Then, this integral part reduces to
\begin{equation}
\label{eq:lagrangian_action}
\mathcal{S}[y] = \int_0^{T} \tilde{\mathcal{L}}(y(t),\dot{y}(t))\, dt\,,
\end{equation}
which is a standard Lagrangian action for a Lagrangian system on $T^*\mathcal{Q}$ and can be readily discretised following \cite{marsden2001discrete}. In particular, one defines the \textit{exact discrete Lagrangian},
\begin{equation}
\label{eq:exact_discrete_new_Lagrangian}
\tilde{\mathcal{L}}_d^e(y_k,y_{k+1},h) = \int_{0}^{h} \tilde{\mathcal{L}}(y_{k,k+1}(t),\dot{y}_{k,k+1}(t))\, dt\,,
\end{equation}
where $y_{k,k+1}: [0,h] \to T^*\mathcal{Q}$ is a solution of the Euler-Lagrange equations \eqref{eulerLagrange_no_u1} and \eqref{eulerLagrange_no_u2} satisfying that $y_{k,k+1}(0) = y_k$ and $y_{k,k+1}(h) = y_{k+1}$. This solution is warranted to exist and be unique for sufficiently small $h$ and sufficiently close $y_k$ and $y_{k+1}$ so the exact discrete Lagrangian is well-defined. Despite it only being defined under such conditions it is customary to extend its domain for simplicity, writing $\tilde{\mathcal{L}}_d^e: T^*\mathcal{Q} \times T^*\mathcal{Q} \times \mathbb{R} \to \mathbb{R}$.\\

Let $T_d = \left\lbrace t_k = kh \;\colon\; k= 0,\dots, N, N h = T\right\rbrace$ be a uniform temporal grid discretising the continuous interval $[0,T]$ and consider a discrete curve $y_d: T_d \to T^*\mathcal{Q}$, i.e. a collection of ordered points $\left\lbrace y_{0}, ..., y_N \right\rbrace$. Then, one constructs a discrete analogue of \eqref{eq:lagrangian_action}
\begin{equation}
\mathcal{S}_d^e[y_d] = \sum_{k = 0}^{N-1} \tilde{\mathcal{L}}_d^e(y_k,y_{k+1},h)\,.
\label{discretActiondef}
\end{equation}
Applying Hamilton's principle to this discrete action with fixed $y_0$ and $y_N$ one obtains the \textit{so-called} discrete Euler-Lagrange equations
\begin{equation}
\label{eq:generic_DEL}
D_{2} \tilde{\mathcal{L}}_d^e(y_{k-1},y_{k},h) + D_{1} \tilde{\mathcal{L}}_d^e(y_{k},y_{k+1},h) = 0\,, \quad k = 1, ..., N-1,
\end{equation}
that characterize the critical points of this action. Remarkably, if $y_d$ is such a critical point, then there exists $y: [0,T] \to T^*\mathcal{Q}$, extremiser of \eqref{eq:lagrangian_action} such that $y_d = y(T_d)$, i.e. $y_d$ is a sampling of this solution over $T_d$.\\

The construction of variational integrators begins by constructing an approximation $\tilde{\mathcal{L}}_d$ to the exact discrete Lagrangian, \eqref{eq:exact_discrete_new_Lagrangian}, and following the same steps in order to produce an approximation of $y_d$. The initial result in \cite{marsden2001discrete} amended in \cite{patrick2009error}, tells us that the convergence rate of the approximation of $y_d$ to it corresponds to the convergence rate of the approximation of $\tilde{\mathcal{L}}_d$ to $\tilde{\mathcal{L}}_d^e$ in the celebrated variational error theorem.\\

In our case, following our continuous approach, we do not intend to fix $y_0$ and $y_N$ and, furthermore, we need to consider the boundary (initial and terminal) costs that appear in \eqref{newAugmentedObjective_no_u}. However, in order to provide a suitable exact discrete analogue of the control-independent $\tilde{\mathcal{J}}[y]$, we need to deal with the correct discretisation of $\dot{q}$. For this, we make use of the fact highlighted in Remark \ref{rmk:pLambdaEqualVelocity}.\\

In the discrete setting, one can define analogues to the fibre derivative provided by the derivatives of the discrete Lagrangian with respect to its first and second arguments. Namely,
\begin{subequations}
\begin{align}
p_{y,k}^{-}(y_{k},y_{k+1},h) &= - D_1 \tilde{\mathcal{L}}_d(y_{k},y_{k+1},h) \in T^*_{y_k} T^* \mathcal{Q}\,,\\
p_{y,k+1}^{+}(y_{k},y_{k+1},h) &= \hphantom{-} D_2 \tilde{\mathcal{L}}_d(y_{k},y_{k+1},h) \in T^*_{y_{k+1}} T^* \mathcal{Q}\,.
\end{align}
\label{eq:discreteMomentadef}
\end{subequations}
These allow us to interpret \eqref{eq:generic_DEL} as a matching of momenta, i.e. $p_{y,k}^{+}(y_{k-1},y_{k},h) = p_{y,k}^{-}(y_{k},y_{k+1},h)$, which implies that there exist unique well-defined momenta $p_{y,k}$ along extremals. Moreover, these can be put into direct correspondence to the canonical momenta obtained in the continuous setting, i.e.~\eqref{contLegendre}, so that $p_{y,k} = (p_{q,k}, p_{\lambda,k})$ with
\begin{subequations}
	\label{eq:discreteMomenta}
	\begin{align}
        p_{q,k} &= v_{\lambda,k} + D_2 f(q_{k},v_{k})^{\top} \lambda_k\,,\\
        p_{\lambda,k} &= v_k\,,
    \end{align}
\end{subequations}
which gives us appropriate values of $v_k$ and $v_{\lambda,k}$. In light of this, and focusing on the former, these expressions allow us to write
\begin{subequations}
	\label{eq:discreteVelocities}
	\begin{align}
		v_{k}^{-}(y_{k},y_{k+1},h) &= - D_{\lambda_k} \tilde{\mathcal{L}}_d(y_{k},y_{k+1},h) \in T_{q_k} \mathcal{Q}\,,\\
		v_{k+1}^{+}(y_{k},y_{k+1},h) &= \hphantom{-} D_{\lambda_{k+1}} \tilde{\mathcal{L}}_d(y_{k},y_{k+1},h) \in T_{q_{k+1}} \mathcal{Q}\,.
    \end{align}
\end{subequations}
If the discrete Lagrangians are indeed the exact discrete ones, then we write $v_{k}^{e,-}(y_{k},y_{k+1},h)$ and $v_{k+1}^{e,+}(y_{k},y_{k+1},h)$. With this, we can finally give the following

\begin{definition}
We say
\begin{align}
\tilde{\mathcal{J}}^e_d[y_d,\mu,\nu] &= \phi(q_N,v_{N}^{e,+}(y_{N-1},y_{N},h)) + \mu^{\top}(q_0 - q^0) + \nu^{\top} (v_{0}^{e,-}(y_{0},y_{1},h) - \dot{q}^0)\nonumber\\
&+ \lambda_N^{\top} v_{N}^{e,+}(y_{N-1},y_{N},h) - \lambda_0^{\top} v_{0}^{e,-}(y_{0},y_{1},h) - \sum_{k=0}^{N-1} \tilde{\mathcal{L}}_d^e(y_{k},y_{k+1},h)\,.\label{eq:exactDiscreteobjective_no_u}
\end{align}
is the exact discrete new control-independent augmented objective function of $\tilde{\mathcal{J}}[y]$.
\end{definition}

Stationarity of this objective \eqref{eq:exactDiscreteobjective_no_u}, $\delta\tilde{\mathcal{J}}^e_d[y_d]=0 $, with the short-hand notation $\tilde{\mathcal{L}}^e_d(y_k,y_{k+1},h)  \equiv  \tilde{\mathcal{L}}^e_{d,k}$, leads to the following optimality conditions:
\begin{subequations}
\label{exact_no_u_boundaries}
\begin{align}
    \text{(Euler-Lagrange)}&&\delta y_k&: \quad D_2 \tilde{\mathcal{L}}_{d,k-1}^{e} + D_1 \tilde{\mathcal{L}}_{d,k}^{e} = 0 , \qquad \text{for } k=2,\dots,N-2\label{yexactOptimalitynou}\\
    \text{(boundary conditions)}&&\delta q_0 &: \quad (\nu - \lambda_0)^\top D_{q_0} v_0^{e,-} + \mu^\top - D_{q_0} \tilde{\mathcal{L}}^{e}_{d,0} = 0\label{exactboundarynou1}\\
    &&\delta \lambda_0 &: \quad  (\nu - \lambda_0)^\top D_{\lambda_0} v_0^{e,-} - v_0^{e,-} - D_{\lambda_0} \newL_{d,0}^{e} = 0 \label{exactboundarynou2}\\
    &&\delta y_1 &: \quad (\nu - \lambda_0)^\top D_2 v_0^{e,-} - (D_2 \tilde{\mathcal{L}}_{d,0}^{e} + D_1 \tilde{\mathcal{L}}_{d,1}^{e}) = 0 \label{exactboundarynou3}\\
    &&\delta y_{N-1} &: \quad  \left(D_2 \phi + \lambda_N^\top\right) D_1 v_N^{e,+} - (D_2 \tilde{\mathcal{L}}_{d,N-2}^{e} + D_1 \tilde{\mathcal{L}}_{d,N-1}^{e})= 0\label{exactboundarynou4}\\
    &&\delta q_{N} &: \quad  \left(D_2 \phi + \lambda_N^\top\right) D_{q_N} v_N^{e,+} + D_1 \phi - D_{q_N} \tilde{\mathcal{L}}_{d,N-1}^{e} = 0\label{exactboundarynou5}\\
    &&\delta \lambda_{N} &: \quad \left(D_2 \phi + \lambda_N^\top\right) D_{\lambda_N} v_N^{e,+} + v_N^{e,+}- D_{\lambda_N} \newL_{d,N-1}^e = 0 \label{exactboundarynou6}\\
    &&\delta \mu &: \quad  q_0 = q^0\label{exacnoutboundarymu}\\
    &&\delta \nu &: \quad v_0^{e,-} = \dot q ^0.\label{exacnoutboundarynu}
\end{align}
\end{subequations}

\begin{remark}
    \label{rmk:discrete_velocity_lambda}
    Notice that, in contrast with the necessary conditions for optimality found in the continuous setting, \eqref{eulerLagrangeufull}, no explicit variation with respect to either $\dot{q}(0)$, \eqref{boundaryend2}, or $\dot{q}(T)$, \eqref{boundaryend4}, appears here while variations with respect to $\lambda_0 \equiv \lambda(0)$, \eqref{exactboundarynou2} and $\lambda_N \equiv \lambda(T)$, \eqref{exactboundarynou6}, do appear naturally due to $y = (q,\lambda)$ being our primary variables, in contrast with the usual $(q,\dot{q})$. Nevertheless, it is still possible to recuperate Equations \eqref{boundaryend2} and \eqref{boundaryend4} simply by differentiating $\tilde{\mathcal{J}}_d^e$ with respect to $v_0^{-}$ and $v_N^+$. This is canonically justifiable since
    \begin{equation*}
        \lambda_N^{\top} v_{N}^{e,+}(y_{N-1},y_{N},h) - \lambda_0^{\top} v_{0}^{e,-}(y_{0},y_{1},h) - \sum_{k=0}^{N-1} \tilde{\mathcal{L}}_d^e(y_{k},y_{k+1},h)
    \end{equation*}
    may itself be regarded as a recast of a discrete cost function $C_d^e(q_0,v_0,q_N,v_N,h)$ with $\lambda_0$ and $\lambda_N$ coinciding with its canonical momenta associated to $v_0$ and $v_N$.
    In any case, the role of the resulting equations is equivalent, as will become clear in the proof of the coming theorem.
\end{remark}

At first glance, the optimality conditions \eqref{exact_no_u_boundaries} derived from the exact discrete objective function $\delta\tilde{\mathcal{J}}^e_d[y_d]=0 $ may appear to be quite different from the continuous optimality conditions \eqref{boundaryend1}-\eqref{boundaryend4} and \eqref{eulerLagrange_no_u} derived for the continuous objective $\tilde{\mathcal{J}}$ in \eqref{newAugmentedObjective_no_u}. However, they are, in fact, equivalent which is shown in the following

\begin{theorem}
    \label{thm:equivalence_no_u}
    Assume that either
    \begin{enumerate}[label=H\arabic*]
        \item \label{itm:hyp1} $D_{\lambda_0} v_0^{e,-}$ and $D_{\lambda_N} v_N^{e,+}$ are of full rank, or
        \item \label{itm:hyp2} \eqref{exactboundarynou2} and \eqref{exactboundarynou6}  are substituted by \eqref{boundaryend2} and \eqref{boundaryend4}, respectively.
    \end{enumerate}
    Then, the necessary conditions for optimality derived from the extremisation of $\tilde{\mathcal{J}}$, \eqref{newAugmentedObjective_no_u}, and those of $\tilde{\mathcal{J}}_d^e$,  \eqref{eq:exactDiscreteobjective_no_u} are equivalent. This means both that:
    \begin{itemize}
        \item given $(y,\mu,\nu)$ optimal, i.e. solution of \eqref{eulerLagrangeufull}, then $(y_d = y(T_d),\mu,\nu)$ is discretely optimal, i.e. it solves \eqref{exact_no_u_boundaries}, for all $T_d$;
        \item given $T_d = \left\lbrace t_k\right\rbrace_0^N$ and $(y_d,\mu,\nu)$ discretely optimal, then the piecewise-defined $y$ s.t. $\left.y(t)\right\vert_{[t_k,t_{k+1}]} = y_{k,k+1}(t - t_k)$, $k = 0,\dots, N-1$ where $y_{k,k+1}$ is the solution of \eqref{eulerLagrange_no_u1} and \eqref{eulerLagrange_no_u2} with $y_{k,k+1}(0) = y_k$ and $y_{k,k+1}(t_{k+1}-t_k) = y_{k+1}$, is optimal.
    \end{itemize}
\end{theorem}

\begin{proof}
    \label{proofexactequivalence_no_u}
    The equivalence of the discrete \eqref{yexactOptimalitynou} and the continuous Euler-Lagrange equations \eqref{eulerLagrange_no_u} for $[t_1,t_{N-1}],$ is given by the construction of the exact discrete action \eqref{discretActiondef}, see \cite{marsden2001discrete}.\\
    For the equivalence of the boundary conditions, let us first consider \eqref{exactboundarynou2} and \eqref{exactboundarynou6}. By \eqref{eq:discreteVelocities}, we see that only the first term of each equation survives. If $D_{\lambda_0} v_0^{e,-}$ and $D_{\lambda_N} v_N^{e,+}$ are of full-rank (\ref{itm:hyp1}), then these equations are already equivalent to \eqref{boundaryend2} and \eqref{boundaryend4}. Otherwise, we may apply Hypothesis \ref{itm:hyp2} since these equations only fix the values of $\lambda_0$ and $\lambda_N$.\\
    Either way, Equations \eqref{exactboundary1},\eqref{exactboundary3},\eqref{exactboundary4} and \eqref{exactboundary5} reduce to
    
    \begin{subequations}
    \begin{align}
    \delta q_0 &: \quad \mu^\top - D_{q_0} \tilde{\mathcal{L}}^{e}_{d,0} = 0\label{exactboundarynousimp1}\\
    \delta y_1 &: \quad (D_2 \tilde{\mathcal{L}}_{d,0}^{e} + D_1 \tilde{\mathcal{L}}_{d,1}^{e}) = 0 \label{exactboundarynousimp3}\\
    \delta y_{N-1} &: \quad  (D_2 \tilde{\mathcal{L}}_{d,N-2}^{e} + D_1 \tilde{\mathcal{L}}_{d,N-1}^{e})= 0\label{exactboundarynousimp4}\\
    \delta q_{N} &: \quad  D_1 \phi - D_{q_N} \tilde{\mathcal{L}}_{d,N-1}^{e} = 0\label{exactboundarynousimp5}
\end{align}
\label{simplifiedboundaries}
\end{subequations}
Equations \eqref{exactboundarynousimp3} and \eqref{exactboundarynousimp4} are simply the missing discrete Euler-Lagrange equations for the initial and final intervals respectively.
Under the canonical identifications provided by \eqref{eq:discreteMomenta}, and the relations highlighted in Remark \ref{rmk:OCP_first_order}, \eqref{exactboundarynousimp1} and \eqref{exactboundarynousimp5} give us
\begin{align*}
    - D_{q_0} \tilde{\mathcal{L}}_{d,0}^e &= p_{q,0} = \mu^\top = - \lambda_{q,0}^\top\\
    D_{q_N} \tilde{\mathcal{L}}_{d,{N-1}}^e &= p_{q,N} = D_1 \phi = - \lambda_{q,N}^\top
\end{align*}
which concludes the proof.
\end{proof}

\begin{remark}
    \label{rmk:on_equivalence_hypotheses}
    It is important to highlight that, first,
    even if Hypothesis \ref{itm:hyp1} did not hold and \ref{itm:hyp2} was not applied, 
    Equations \eqref{exact_no_u_boundaries} would still provide the correct solution, though $\lambda_0$ and $\lambda_N$ would not be fully defined. This will become clear from our analysis of the control-dependent case in the next section (see in particular Remark \ref{rmk:uind_equations_well_defined}).\\
    Second, we believe that Hypothesis \ref{itm:hyp1} in Theorem \ref{thm:equivalence_no_u} is satisfied in the exact case under the assumption that the system under consideration, \eqref{SODE} in this case, be small-time locally controllable along the optimal trajectory. However, we do not have a proof of this yet. Besides, it is important to note that Hypothesis \ref{itm:hyp1} is not the norm in application. Indeed, when approximating the exact discrete Lagrangian, it is often the case that the resulting matrices $D_{\lambda_0} v_0^{-}$ and $D_{\lambda_N} v_N^{+}$ be rank deficient. Thus, the application of Hypothesis \ref{itm:hyp2} should not be regarded as an ad-hoc solution and it is in fact typical and, moreover, canonical in the sense of Remark \ref{rmk:discrete_velocity_lambda}. This will also become clearer in the control-dependent case (see in particular Remark \ref{rmk:discrete_velocity_lambda_with_controls}).
\end{remark}

\subsection{Exact semi-discrete setting}
\label{secsemidiscrete}

Motivated by the definitions of the exact discrete Lagrangian, we present here an exact discrete new control-dependent Lagrangian. We propose the following definition and assume it being well-defined.
\begin{definition}
\label{def:newControlDepDiscLag}
Consider the functional
\begin{equation}
\label{eq:exact_semidiscrete_new_Lagrangian}
\tilde{\mathcal{L}}_d^{\mathcal{E},e}[u_{k,k+1}](y_k,y_{k+1},h) \equiv \tilde{\mathcal{L}}_d^{\mathcal{E},e}(y_k,y_{k+1},u_{k,k+1},h) = \int_{0}^{h} \tilde{\mathcal{L}}^{\mathcal{E}}(y_{k,k+1}(t),\dot{y}_{k,k+1}(t),u_{k,k+1}(t))\, dt\,.
\end{equation}
where $u_{k,k+1} \in C^1([0,h],\mathcal{N})$\footnote{Actually, from a geometric point of view the control argument of $\tilde{\mathcal{L}}_d^{\mathcal{E},e}$ should be considered a curve $u_{k,k+1} \in C^{1}([0,h],\mathcal{E})$, such that $\pi^{\mathcal{E}} \circ u_{k,k+1} = q_{k,k+1} = \pi_\mathcal{Q} \circ y_{k,k+1}$, $q_{k,k+1}(0) = \pi_\mathcal{Q}(y_{k})$ and $q_{k,k+1}(h) = \pi_\mathcal{Q}(y_{k+1})$, but we abuse the identification with its fibres for simplicity.} is any control curve that steers the system from $y_k$ to $y_{k+1}$ in time $h$, and $y_{k,k+1}$ is the solution of \eqref{eulerLagrangeu1} and \eqref{eulerLagrangeu2} under said control so that $y_{k,k+1}(0) = y_{k}$ and $y_{k,k+1}(h) = y_{k+1}$. We say $\tilde{\mathcal{L}}_d^{\mathcal{E},e}$ is the exact discrete new control-dependent Lagrangian for the OCP.
\end{definition}

Clearly, this is a functional dependent on the function $u_{k,k+1}$. Therefore, it is also appropriate to dub this an exact \textit{semi-discrete} new Lagrangian for the OCP. 
One could argue that the restriction on $u_{k,k+1}$ should make it dependent on the tuple $(y_k,y_{k+1},h)$. However, one can think of it the other way around: It is not so much that the tuple determines the control as that the control determines the tuple. In particular, if we fix $y_k$ and $h$, every control curve will lead the system from $y_k$ to some $\tilde{y}_{k+1}$ in time $h$, but we simply restrict to those where $\tilde{y}_{k+1} = y_{k+1}$, assuming $y_{k+1}$ is inside the reachable set.
\\

\begin{remark}
From the definition, if $y_{k,k+1} = (q_{k,k+1},\lambda_{k,k+1})$, then,
\begin{align*}
q_{k,k+1}(t) &= q_{k,k+1}(q_{k}, q_{k+1}, u_{k,k+1}, h; t)\,,\\
\lambda_{k,k+1}(t) &= \lambda_{k,k+1}(q_{k}, \lambda_{k}, q_{k+1}, \lambda_{k+1}, u_{k,k+1}, h; t)\,, \qquad \forall t \in [0,h].
\end{align*}
In particular, notice that $q_{k,k+1}$ cannot depend on $\lambda_k$ or $\lambda_{k+1}$ since \eqref{eulerLagrangeu2} is completely independent of $\lambda$. This will be extraordinarily important in the analysis of the resulting (semi-)discrete necessary conditions for optimality.
\end{remark}

Whenever $\bar{u}_{k,k+1}$ is optimal, meaning it satisfies the minimisation condition \eqref{eulerLagrangeu3}, we can express it in terms of $q_{k,k+1}$, $\lambda_{k,k+1}$, resulting in $\bar{u}_{k,k+1}(y_{k},y_{k+1},h;t)$. Thus, from their respective definitions we obtain
\begin{equation}
    \label{eq:dep_indep_identity}
    \newL_d^e(y_k,y_{k+1},h) = \newL_d^{\mathcal{E},e}(y_k,y_{k+1},\bar{u}_{k,k+1},h).
\end{equation}

\begin{proposition}
\label{prp:UDiscreteFibreDer}
Consider the exact discrete new control-dependent Lagrangian \eqref{eq:exact_semidiscrete_new_Lagrangian} for \eqref{problem_definition}. Then,
\begin{align*}
p_{y,k}^{-}(y_{k},y_{k+1},u_{k,k+1},h) &= - D_1 \tilde{\mathcal{L}}_d^{\mathcal{E},e}(y_k,y_{k+1},u_{k,k+1},h) = D_2 \tilde{\mathcal{L}}^{\mathcal{E}}(y_{k,k+1}(0),\dot{y}_{k,k+1}(0),u_{k,k+1}(0))\,,\\
p_{y,k+1}^{+}(y_{k},y_{k+1},u_{k,k+1},h) &= \hphantom{-} D_2 \tilde{\mathcal{L}}_d^{\mathcal{E},e}(y_k,y_{k+1},u_{k,k+1},h) = D_2 \tilde{\mathcal{L}}^{\mathcal{E}}(y_{k,k+1}(h),\dot{y}_{k,k+1}(h),u_{k,k+1}(h))\,.
\end{align*}
\end{proposition}

\begin{proof}
The proof of this is essentially the same as for a standard discrete Lagrangian. In particular, if we focus on $p_{y,{k+1}}$, we get
\begin{align*}
D_2 \tilde{\mathcal{L}}_d^{\mathcal{E},e}(y_k,y_{k+1},u_{k,k+1},h) &= \frac{\partial}{\partial y_{k+1}} \int_{0}^{h} \tilde{\mathcal{L}}^{\mathcal{E}}(y_{k,k+1}(t),\dot{y}_{k,k+1}(t),u_{k,k+1}(t))\, dt\\
&= \int_{0}^{h} (D_{\text{EL}}\tilde{\mathcal{L}}^{\mathcal{E}})(y_{k,k+1}(t),\dot{y}_{k,k+1}(t),\ddot{y}_{k,k+1}(t),u_{k,k+1}(t),\dot{u}_{k,k+1}(t)) \frac{\partial y_{k,k+1}}{\partial y_{k+1}} \, dt\\
&+ \int_{0}^{h} D_3 \tilde{\mathcal{L}}^{\mathcal{E}}(y_{k,k+1}(t),\dot{y}_{k,k+1}(t),u_{k,k+1}(t)) \frac{\partial u_{k,k+1}}{\partial y_{k+1}} \, dt\\
&+ \left[ D_2 \tilde{\mathcal{L}}^{\mathcal{E}}(y_{k,k+1}(t),\dot{y}_{k,k+1}(t),u_{k,k+1}(t)) \frac{\partial y_{k,k+1}}{\partial y_{k+1}}\right]_0^h \,
\end{align*}
Here, $D_{\text{EL}}$ denotes the Euler-Lagrange operator which produces Equations~\eqref{eulerLagrangeu1}-\eqref{eulerLagrangeu2}\footnote{Though not important for this proof, it is worth mentioning that for \eqref{problem_definition}, $\dot{u}_{k,k+1}$ does not appear in these equations.}. Thus, by the definition of $\tilde{\mathcal{L}}_d^{\mathcal{E},e}$ the first integral term of the last equality vanishes. The second integral term of the last equality also vanishes since $u_{k,k+1}$ is assumed an independent variable. Finally, the last term is composed of the difference of two evaluations. By the mutual independence of $y_k$, $y_{k+1}$,
\begin{equation*}
\frac{\partial y_{k,k+1}}{\partial y_{k+1}}(0) = \frac{\partial y_{k}}{\partial y_{k+1}} = 0\,, \qquad \frac{\partial y_{k,k+1}}{\partial y_{k+1}}(h) = \frac{\partial y_{k+1}}{\partial y_{k+1}} = I\,,
\end{equation*}
we get the corresponding result. The case $p_{y,k}$ follows similarly.
\end{proof}

By \eqref{eq:dep_indep_identity}, as a trivial corollary, when $u_{k,k+1}$ satisfies \eqref{eulerLagrangeu3}, these coincide with those of the control-independent case. Also, Proposition~\eqref{prp:UDiscreteFibreDer} shows that we have access to analogous constructions to \eqref{eq:discreteVelocities}, i.e.
\begin{subequations}
	\label{eq:discreteVelocitiesU}
	\begin{align}
		v_{k}^{-} &= - D_{\lambda_k} \tilde{\mathcal{L}}_d^{\mathcal{E}}(y_{k},y_{k+1},u_{k,k+1},h)\,,\\
		v_{k+1}^{+} &= \hphantom{-} D_{\lambda_{k+1}} \tilde{\mathcal{L}}_d^{\mathcal{E}}(y_{k},y_{k+1},u_{k,k+1},h)\,,
    \end{align}
\end{subequations}
are well-defined. Moreover, from our definition of $\tilde{\mathcal{L}}_d^{\mathcal{E},e}$, it follows that in the exact case these control-dependent velocities only depend on $q_{k}$, $q_{k+1}$, $u_{k,k+1}$ and $h$, since
\begin{subequations}
\begin{align}
	v_{k}^{e,-}(q_{k},q_{k+1},u_{k,k+1},h) &= \left.\frac{d}{dt}\right\vert_{t = 0} q_{k,k+1}(q_{k}, q_{k+1}, u_{k,k+1}, h; t)\,,\\
	v_{k+1}^{e,+}(q_{k},q_{k+1},u_{k,k+1},h) &= \left.\frac{d}{dt}\right\vert_{t = h} q_{k,k+1}(q_{k}, q_{k+1}, u_{k,k+1}, h; t)\,.
\end{align}
\label{eq:exactdiscreteVelocitiesU}
\end{subequations}

\begin{remark}
    \label{rmk:discrete_velocity_lambda_with_controls}
    This explains what was mentioned in Remark \ref{rmk:on_equivalence_hypotheses}. $v_0^{-}$ and $v_N^{+}$ depend on $\lambda_0$ and $\lambda_N$, respectively, solely through the controls. 
    Also,
    the dependence of $u_{0,1}$ and $u_{N-1,N}$ on these adjoints is non-local in the sense that they are the result of solving a two-point boundary value problem. Thus, the resulting dependence may appear at high 
    powers of $h$, and approximate discretisations may or may not be able to capture the correct behaviour of $D_{\lambda_0} v_0^{-}$ and $D_{\lambda_N} v_N^{+}$ in the control-independent case, hence the need to impose Hypothesis \ref{itm:hyp2} in those cases to uniquely determine these costates.
\end{remark}


With these, we can provide the following
\begin{definition}
Let $u_d$ be an ordered collection of single-interval controls, i.e. $u_d = \left\lbrace u_{0,1}, ..., u_{N-1,N}\right\rbrace$. Then, we say
\begin{align}
\tilde{\mathcal{J}}^{\mathcal{E},e}_d[y_d,u_d,\mu,\nu] &= \phi(q_N,v_{N}^{e,+}(q_{N-1},q_{N},u_{N-1,N},h))\nonumber\\
&+ \lambda_N^{\top} v_{N}^{e,+}(q_{N-1},q_{N},u_{N-1,N},h) - \lambda_0^{\top} v_{0}^{e,-}(q_{0},q_{1},u_{0,1},h)\nonumber\\
&+ \mu^{\top}(q_0 - q^0) + \nu^{\top} (v_{0}^{e,-}(q_{0},q_{1}, u_{0,1} ,h) - \dot{q}^0) - \sum_{k=0}^{N-1} \tilde{\mathcal{L}}_d^{\mathcal{E},e}(y_{k},y_{k+1},u_{k,k+1},h)\,.
\label{eq:exactDiscreteobjective}
\end{align}
is the exact discrete new control-dependent augmented objective function of $\tilde{\mathcal{J}}^{\mathcal{E}}[y,u,\mu,\nu]$.
\end{definition}

In order to obtain the necessary conditions for optimality, we make use of the following
\begin{proposition}
    \label{prp:control_dep_lambda_indep}
    By construction, $\tilde{\mathcal{J}}_d^{\mathcal{E},e}$, \eqref{eq:exactDiscreteobjective} is independent of $\lambda_0$ and $\lambda_N$.
\end{proposition}

\begin{proof}
    Since the velocities \eqref{eq:exactdiscreteVelocitiesU} are independent of $\lambda$, variation of $\tilde{\mathcal{J}}_d^{\mathcal{E},e}$ with respect to $\lambda_0$ and $\lambda_N$ reduces to
    \begin{align*}
        \delta \lambda_0 &:\quad - v_0^{e,-} - D_{\lambda_0} \tilde{\mathcal{L}}^{\mathcal{E},e}_{d,0} = 0,\\
        \delta \lambda_{N} &: \quad  v_N^{e,+}- D_{\lambda_N} \tilde{\mathcal{L}}^{\mathcal{E},e}_{d,N-1} = 0.
    \end{align*}
    However, by \eqref{eq:discreteVelocitiesU} these are identically zero proving our claim.
\end{proof}

The previous result may seem surprising due to $\lambda_0$ and $\lambda_N$ featuring prominently in the definition of $\tilde{\mathcal{J}}_d^{\mathcal{E},e}$. However, this is so by the very definition of the augmented cost function \eqref{newAugmentedObjective}. The extra terms obtained in the process of integration by parts, $\lambda(T)^{\top} \dot{q}(T)$ and $\lambda(0)^{\top} \dot{q}(0)$, that led us to this functional are precisely there to annul the variations with respect to these adjoints, making the overall expression equivalent to \eqref{augmentedObjective}. Terms $\delta \lambda(0)$ and $\delta \lambda(N)$ do not appear in the variation of this latter functional.

The necessary conditions for optimality for the exact discrete control-dependent objective function are once more derived by setting $\delta \newJ_d^{\mathcal{E},e}=0,$ leading to 
\begin{subequations}
\begin{align}
        \text{(Euler-Lagrange)}&&\delta y_k&: \quad D_1 \tilde{\mathcal{L}}_{d,k}^{\mathcal{E},e} + D_2 \tilde{\mathcal{L}}_{d,k-1}^{\mathcal{E},e}= 0, \quad \text{for } k=2,\dots,N-2,\label{yexactOptimality}\\
    \text{(minimisation)}&&\delta u_{k,k+1}&:\quad  D_3 \tilde{\mathcal{L}}_{d,k}^{\mathcal{E},e}[\delta u_{k,k+1}] = 0, \hspace{0.97cm}\text{for } k=2,\dots,N-2,\label{uExactOptimality}\\
    \text{(boundary conditions)}&&\delta q_0&:\quad (\nu-\lambda_0 )^\top D_{q_0} v_0^{e,-} + \mu^\top - D_{q_0} \tilde{\mathcal{L}}^{\mathcal{E},e}_{d,0} = 0,\label{exactboundary1}\\
      &&\delta y_1&:\quad  (\nu-\lambda_0)^\top D_2 v_{0}^{e,-} - (D_1 \tilde{\mathcal{L}}_{d,1}^{\mathcal{E},e} + D_2 \tilde{\mathcal{L}}_{d,0}^{\mathcal{E},e})= 0,\label{exactboundary3}\\
       &&  \delta y_{N-1}&:\quad \left(D_2 \phi +\lambda_N ^\top\right) D_1 v_N^{e,+} - (D_2 \tilde{\mathcal{L}}_{d,N-2}^{\mathcal{E},e}+D_1 \tilde{\mathcal{L}}_{d,N-1}^{\mathcal{E},e}) = 0,\label{exactboundary4}\\
       &&\delta q_{N}&:\quad  \left(D_2 \phi+\lambda_N^\top\right) D_{q_N} v_N^{e,+}  + D_1\phi - D_{q_N} \tilde{\mathcal{L}}_{d,N-1}^{\mathcal{E},e} = 0,\label{exactboundary5}\\
      && \delta u_{0,1}&:\quad (\nu -\lambda_0)^\top D_{3} v_0^{e,-}[\delta u_{0,1}] - D_{3} \tilde{\mathcal{L}}_{d,0}^{\mathcal{E},e}[\delta u_{0,1}]  = 0,\label{exactboundary7}\\
      && \delta u_{N-1,N}&:\quad \left(D_2 \phi + \lambda_N^\top\right) D_{3} v_N^{e,+}[\delta u_{N-1,N}] - D_{3} \tilde{\mathcal{L}}_{d,N-1}^{\mathcal{E},e}[\delta u_{N-1,N}]  = 0,\label{exactboundary8}\\
      && \delta \mu&:\quad q_0 = q^0, ~~~~ \label{exactboundary9}\\
     &&  \delta \nu&:\quad v_0^{e,-} = \dot q ^0.\label{exactboundary}
\end{align}
\label{exactboundaries}
\end{subequations}

where
\begin{equation}
    \label{eq:exact_control_variations}
    D_3 \tilde{\mathcal{L}}_{d,k}^{\mathcal{E},e}[\delta u_{k,k+1}] = \int_{0}^h D_3 \tilde{\mathcal{L}}^\mathcal{E}(y_{k,k+1},\dot{y}_{k,k+1},u_{k,k+1})\, \delta u_{k,k+1} \,dt
\end{equation}

By Proposition \ref{prp:control_dep_lambda_indep}, neither \eqref{eq:exactDiscreteobjective} nor any of the equations in \eqref{exactboundaries} actually depend on $\lambda_0$ and $\lambda_N$. In order to fix their values, we may introduce the following Hypothesis
\begin{enumerate}[label=H\arabic*',start=2]
    \item \label{itm:H2p} Equations \eqref{boundaryend2} and \eqref{boundaryend4} are appended to the system. 
\end{enumerate}
As mentioned in Remark \ref{rmk:discrete_velocity_lambda} this is canonical. The equality between continuous and discrete optimality conditions in the control-dependent case is given by the following

\begin{theorem}
\label{equivalencesemidiscrete}
    The discrete optimality conditions \eqref{exactboundaries} together with Hypothesis \ref{itm:H2p}
    are equivalent to the continuous control-dependent optimality conditions \eqref{eulerLagrangeufull}.
\end{theorem}
\begin{proof}
    The proof is mostly the same as for the independent case, Theorem \ref{proofexactequivalence_no_u} under Hypothesis \ref{itm:hyp2}.
    In particular, \eqref{boundaryend2} and \eqref{boundaryend4} lead to the same simplification of the boundary optimality conditions \eqref{exactboundaries} as was shown in theorem \ref{proofexactequivalence_no_u}. Consequently, the equivalence to the transversality conditions \eqref{boundaryend1}-\eqref{boundaryend4} and state/costate dynamics \eqref{eulerLagrangeu1},\eqref{eulerLagrangeu2} is analogously given by construction of the discrete Lagrangian $\newL_d^{\mathcal{E},e}$.
    The difference with Theorem \ref{proofexactequivalence_no_u} lies in the variations with respect to the controls. However, notice that \eqref{eq:exact_control_variations} together with the exact discrete minimisation condition \eqref{uExactOptimality} in conjunction with the boundary minimisation conditions \eqref{exactboundary7},\eqref{exactboundary8} yield the continuous minimisation condition \eqref{eulerLagrangeu3} for all $t \in [0,T]$.
\end{proof}

\begin{remark}
\label{rmk:uind_equations_well_defined}
It is important to highlight that the imposition of Hypothesis \ref{itm:H2p} is only necessary in order to prove the equivalence with the continuous equations. Equations \eqref{exactboundaries} form a well-determined set of equations since by Proposition \ref{prp:control_dep_lambda_indep}, $\lambda_0$ and $\lambda_N$ are not actually part of the variable set to begin with. For instance, Equations \eqref{exactboundary1} and \eqref{exactboundary3} readily allow us to obtain $\mu$ and $\nu$ in terms of $q_0$, $q_1$, $q_2$, $\lambda_1$, $\lambda_2$, $u_{0,1}$ and $u_{1,2}$.\\
Since \eqref{exactboundaries} are equivalent to \eqref{exact_no_u_boundaries} with \eqref{exactboundarynou2} and \eqref{exactboundarynou6} removed, this shows that the first point in Remark \eqref{rmk:on_equivalence_hypotheses} is true.
\end{remark}

\section{Optimality conditions for a low-order integration scheme}


\label{loworderapproxsection}

While the exact discrete treatment of the optimal control problems considered is of theoretical relevance, it is not generally usable in applications. 
For these, we rely on approximations which lead to readily usable integration schemes. 
In the following a general low-order integration scheme is used to derive discrete optimality conditions that can be used to apply the new approach to complex examples. The advantage of deriving discrete optimality conditions by approximating the exact Lagrangians lies in the fact that these naturally yield variational integrators.

\subsection{Approximate optimality conditions in the control-independent case}

Let us approximate the exact discrete control-independent Lagrangian 
when $\mathcal{Q} = \mathbb{R}^n$, $n \in \mathbb{N}$, as follows:
\begin{subequations}
\begin{align}
    \newL_{d,k}^e 
    \approx
    \tilde{\mathcal{L}}_{d,k} &=  h\alpha \tilde{\mathcal{L}}_{k}^\gamma + h (1-\alpha)\tilde{\mathcal{L}}_{k}^ {(1-\gamma)} \\
    \tilde{\mathcal{L}}_{k}^ \gamma&=\newL(\overline{y}_k^\gamma,\Delta y_k) = \Delta \lambda_k^\top \Delta q_k + {\overline{\lambda}_k^ {\gamma}}^{T} f_k^\gamma + \frac{1}{2} {\overline{\lambda}_k^{\gamma}}^\top b_k^\gamma  \overline{\lambda}_k^\gamma,\label{approxLagrangian_b}
\end{align}
    \label{approxLagrangian}
\end{subequations}
with $\alpha,\gamma \in [0,1]$ which define the particular low-order integration scheme, and the short-hand notations $\overline{y}_k^\gamma = \gamma y_k + (1-\gamma) y_{k+1}$,  $\Delta y_k = (y_{k+1}-y_k)/h$, and $f_k^\gamma = f\left( \overline{q}_k^{\gamma} ,  \Delta q_{k}\right)$, $b_k^\gamma = b(\overline{q}_k^\gamma)$.

The boundary velocities in terms of the discrete variables  defined in the exact discrete context \eqref{eq:discreteVelocities} can be calculated explicitly now for  $\tilde{\mathcal{L}}_d$
\begin{subequations}
\begin{align}
    v_0^{e,-} &\approx v_0^{-} =  \Delta q_0 -h\alpha\gamma\left(  {f_0^\gamma}+ b_0^\gamma {\overline{\lambda}_0^\gamma}\right) - h(1-\alpha)(1-\gamma)\left({f_0^{(1-\gamma)}}+ b_0^{(1-\gamma)} {\overline{\lambda}_0^{(1-\gamma)}}\right)\label{apV1}\\
     v_N^{e,+} &\approx v_N^{+}=  \Delta q_{N-1} +h\alpha(1-\gamma)\left(  {f_{N-1}^\gamma}+ b_{N-1}^\gamma  {\overline{\lambda}_{N-1}^\gamma} \right) + h(1-\alpha)\gamma\left({f_{N-1}^{(1-\gamma)}}+ b_{N-1}^{(1-\gamma)} {\overline{\lambda}_{N-1}^{(1-\gamma)}}\right).\label{apV2}
\end{align}
\label{apV_all}
\end{subequations}

Inserting \eqref{approxLagrangian},\eqref{apV_all} in \eqref{eq:exactDiscreteobjective_no_u} in place of their exact counterparts provides us with the approximate cost function $\tilde{\mathcal{J}}_d$.\\

Stationarity together with Hypothesis \ref{itm:hyp2} leads to the following set of optimality conditions for $k = 1, \dots,  N-1$
\begin{subequations}
\begin{align}
     \delta \lambda_k &:\; \frac{1}{h^2}(q_{k+1}-2q_k + q_{k-1}) = \alpha\gamma\left({f_k^\gamma} +   b_k^\gamma {\overline{\lambda}_{k}^{\gamma}}\right) +(1-\alpha)(1-\gamma)\left({f_k^{(1-\gamma)}} +   b_k^{(1-\gamma)} {\overline{\lambda}_{k}^{(1-\gamma)}}\right)\label{appEL1}\\
     & \hspace{3.1cm}+\alpha(1-\gamma)\left({f_{k-1}^\gamma} +   b_{k-1}^\gamma {\overline{\lambda}_{k-1}^{\gamma}}\right) +(1-\alpha)\gamma\left({f_{k-1}^{(1-\gamma)}} +   b_{k-1}^{(1-\gamma)} {\overline{\lambda}_{k-1}^{(1-\gamma)}} \right)\nonumber\\
     \delta q_k &:\; \frac{1}{h^2}(\lambda_{k+1}-2\lambda_k + \lambda_{k-1}) = \alpha\left[\left(\gamma D_1 f_k^\gamma - \frac{1}{h}D_2 f_k^\gamma\right)^\top  \overline{\lambda}_k^\gamma +\frac{1}{2}(D_{q_k} b_k^\gamma(\overline{\lambda}_k^\gamma,\overline{\lambda}_k^\gamma))^\top\right] \label{appEL2}\\
     &  \hspace{3.1cm}+ \alpha\left[  \left((1-\gamma) D_1 f_{k-1}^\gamma + \frac{1}{h} D_2 f_{k-1}^\gamma \right)^\top\overline{\lambda}_{k-1}^\gamma+\frac{1}{2}( D_{q_k} b_{k-1}^\gamma (\overline{\lambda}_{k-1}^\gamma, \overline{\lambda}_{k-1}^\gamma))^\top\right]\nonumber\\
     &   \hspace{3.1cm}+(1-\alpha)\left[  \left((1-\gamma) D_1 f_k^{(1-\gamma)} - \frac{1}{h} D_2 f_k^{(1-\gamma)}\right)^\top\overline{\lambda}_k^{(1-\gamma)} + \frac{1}{2}\left(D_{q_k} b_k^{(1-\gamma)}(\overline{\lambda}_k^{(1-\gamma)},\overline{\lambda}_k^{(1-\gamma)})\right)^\top \right]\nonumber\\
     & \hspace{3.1cm} + (1-\alpha)\left[ \left(\gamma D_1 f_{k-1}^{(1-\gamma)} + \frac{1}{h}D_2 f_{k-1}^{(1-\gamma)}\right)^\top \overline{\lambda}_{k-1}^{(1-\gamma)}+\frac{1}{2}\left( D_{q_k} b_{k-1}^{(1-\gamma)} (\overline{\lambda}_{k-1}^{(1-\gamma)}, \overline{\lambda}_{k-1}^{(1-\gamma)})\right)^\top\right]\nonumber\\
    \delta q_0 &:\; \mu = h\alpha\left[-\frac{1}{h}\Delta \lambda_0+  \left(\gamma D_1 f_0^\gamma - \frac{1}{h} D_2 f_0^\gamma\right)^\top\overline{\lambda}_0^\gamma+\frac{1}{2}\left( D_{q_0} b_0^\gamma(\overline{\lambda}_0^\gamma,\overline{\lambda}_0^\gamma)\right)^\top\right]\label{approxboundarynoubegin}\\
    &+ h(1-\alpha)\left[-\frac{1}{h}\Delta \lambda_0 + \left((1-\gamma) D_1 f_0^{(1-\gamma)} - \frac{1}{h}D_2 f_0^{(1-\gamma)}\right)^\top\overline{\lambda}_0^{(1-\gamma)} +\frac{1}{2}\left(D_{q_0} b_0^{(1-\gamma)}(\overline{\lambda}_0^{(1-\gamma)},\overline{\lambda}_0^{(1-\gamma)})\right)^\top\right]\nonumber\\
    \delta q_{N} &:\; D_1 \phi(q_N, v_{N}^+) =h\alpha\left[\frac{1}{h} \Delta \lambda_{N-1} + \left((1-\gamma) D_1 f_{N-1}^\gamma +  \frac{1}{h} D_2 f_{N-1}^\gamma\right)^\top \overline{\lambda}_{N-1}^\gamma +\frac{1}{2}\left( D_{q_N} b_{N-1}^\gamma (\overline{\lambda}_{N-1}^\gamma, \overline{\lambda}_{N-1}^\gamma)\right)^\top\right]\nonumber\\
    &+h(1-\alpha)\left[\frac{1}{h}\Delta \lambda_{N-1} +  \left(\gamma D_1 f_{N-1}^{(1-\gamma)} + \frac{1}{h}D_2 f_{N-1}^{(1-\gamma)}\right)^\top \overline{\lambda}_{N-1}^{(1-\gamma)}+\frac{1}{2}\left( D_{q_N} b_{N-1}^{(1-\gamma)} (\overline{\lambda}_{N-1}^{(1-\gamma)}, \overline{\lambda}_{N-1}^{(1-\gamma)})\right)^\top\right]\\
       \delta \mu &:\; q_0 = q^0  \\
       \delta \nu &:\; v_0^{-} = \dot q ^0.\label{approxoutboundary}\\
       \delta v_0^- &:\; \lambda_0 = \nu \label{approxoutboundarypost1}\\
       \delta v_N^+ &:\;  \lambda_N^\top = -D_2 \phi (q_N, v_{N}^+),\label{approxoutboundarypost2}
\end{align}
\label{approxnouboundaries}
\end{subequations}
where $b(\lambda,\lambda) = \lambda^\top b \lambda$.
These necessary optimality conditions derived from the approximate objective $\newJ_d$ are consequently approximations of the exact necesary optimality conditions \eqref{exact_no_u_boundaries} of $\newJ_d^{e}$.
\begin{remark}
Variations with respect to $\lambda_0$ and $\lambda_N$ give

\begin{subequations}
\begin{align}
      \delta \lambda_0&:  (\nu-\lambda_0 )^\top \left[\alpha\gamma^2 b_0^\gamma +(1-\alpha) (1-\gamma)^2 b_0^{(1-\gamma)} \right] = 0, \label{lamJvarapprrox1}\\
       \delta \lambda_{N}&:  \left(\lambda_N +\frac{\partial \phi}{\partial \dot q}^\top\right)^\top \left[\alpha(1-\gamma)^2 b_{N-1}^\gamma +  \gamma^2 b_{N-1}^{(1-\gamma)}\right]   = 0 .\label{lamJvarapprrox2}
\end{align}
\label{lamJvarapprrox}
\end{subequations}
If \ref{itm:hyp1} holds, then the terms in square brackets are matrices of full rank and so, these equations are equivalent to \eqref{approxoutboundarypost1} and \eqref{approxoutboundarypost2} respectively. Thus, Equations \eqref{appEL1} to \eqref{approxoutboundary} also hold. However, as already noted in Remark \ref{rmk:on_equivalence_hypotheses}, this situation is not the norm and the aforementioned matrices 
may not be
of full rank. 
For instance, for $\gamma=1=\alpha$ the matrices 
reduce to $b_0^1,b_{N-1}^0$ respectively which are semi-definite in the underactuated case. Consequently, in that case Hypothesis \ref{itm:hyp1} is not fulfilled, which has the effect of leaving $\lambda_0$ and $\lambda_N$ defined only up to an element of the kernel of the respective matrices. Thus, Hypothesis \ref{itm:hyp2} needs to be applied to uniquely set their value. 
\end{remark}

Equations \eqref{appEL1},\eqref{appEL2} are discrete versions of the continuous second-order differential equations \eqref{eulerLagrange_no_u1},\eqref{eulerLagrange_no_u2}.
Depending on the chosen $\alpha,\gamma$ values, different evaluations of the right-hand side of the continuous differential equations are performed, yielding different schemes.\\

The formulation in terms of discrete Lagrangians lets us 
determine the order of convergence of the 
resulting algorithm through the application of
variational error analysis \cite{marsden2001discrete}.
For the low-order integration scheme considered, the order of convergence is given by the following
\begin{proposition}
\label{convergencelemma}
Assuming sufficient differentiability, the family of low-order integration schemes defined in \eqref{approxLagrangian} yield optimality conditions that approximate the exact discrete optimality conditions \eqref{exact_no_u_boundaries} to order 2 for $\alpha=1/2, \gamma\in [0,1]$ and $\gamma = 1/2, \alpha\in [0,1]$. For any other choice, the optimality conditions are approximated to order 1.
\end{proposition}
\begin{proof}
    The solution curves $y=(q,\lambda)$ of the Euler-Lagrange equations \eqref{eulerLagrange_no_u} may be expanded in the form $y(t_k + h) = y(t_k) + h \dot y(t_k) + (h^2/2) \ddot y(t_k) + O(h^3)$, leading to  the expansion 
\begin{equation}
\tilde{\mathcal{L}}_{d,k}^e = \int_{t_k}^{t_{k+1}}  \tilde{\mathcal{L}}(y,\dot y) \,dt \approx h\tilde{\mathcal{L}}(y_k,\dot{y}_k) +  \frac{h^2}{2} \left( \frac{\partial \tilde{\mathcal{L}}_k}{\partial y}\dot {y}_k +  \frac{\tilde{\mathcal{L}}_k}{\partial \dot{y}}\ddot{y}_k \right)+ O(h^3),
\label{exexpnoU}
\end{equation}
where $y_k = y(t_k)$ and $\tilde{\mathcal{L}}(y(t_k),\dot y(t_k)) = \tilde{\mathcal{L}}_k$ are being used. 
Next, we substitute the solution in the approximate Lagrangian $\tilde{\mathcal{L}}_{d,k}$. This leads to $\overline y_k^\gamma = y_k + h(1-\gamma)\dot y_k + O(h^2)$, $\Delta y_k = \dot y_k + h \ddot{y}_k + O(h^2)$ and ultimately to
\begin{align}
\tilde{\mathcal{L}}_{d,k} &= h\alpha \tilde{\mathcal{L}}(\overline{y}_k^\gamma,\Delta y_k) +h(1-\alpha) \tilde{\mathcal{L}}(\overline{y}_k^{(1-\gamma)},\Delta y_k)\nonumber \\  
&= h\tilde{\mathcal{L}}_k + \frac{h^2}{2}\left( 2[\alpha (1-\gamma) +(1-\alpha) \gamma]\frac{\partial \tilde{\mathcal{L}}_k}{\partial y}\dot{y}_k + \frac{\partial \tilde{\mathcal{L}}_k}{\partial y}\ddot{y}_k \right) + O(h^3). \label{apexpnoU}
\end{align}
Comparing the expansions \eqref{exexpnoU} and \eqref{apexpnoU} shows that the approximate Lagrangian $\tilde{\mathcal{L}}_d$ is always correct to first-order in $h$, making the scheme consistent.
Further, the integration scheme is of second-order for the cases 
\begin{enumerate}
    \item $\gamma = 1/2$ and $\alpha \in [0,1]$
    \item $\alpha = 1/2$ and $\gamma  \in [0,1]$ 
\end{enumerate} 
In the first case, the resulting schemes are independent of $\alpha$,   yielding the same integration scheme regardless of $\alpha$ choice.  In the second case, different integration schemes are being created in dependence of  $\gamma$, yet they all are of second-order, concluding the proof.
\end{proof}

\subsection{Approximate optimality conditions in the control-dependent case}
\label{sec:ApproximateOC_control_dep}

In order to approximate
$\tilde{\mathcal{J}}_d^{\mathcal{E},e}$, besides choosing a quadrature rule, it is necessary to explicitly discretise the controls themselves too,
which depend on the discrete curve $y_d$ and for the family of integration schemes on the discrete control sets $U_d^{(1)} = \left\lbrace U_{k}^{(1)} \approx u(\bar{t}_k^{\beta}) \right\rbrace_{k=0}^{N-1},U_d^{(2)} = \left\lbrace U_{k}^{(2)} \approx u(\bar{t}_k^{(1-\beta)}) \right\rbrace_{k=0}^{N-1}$, with $\overline{t}_k^\beta = \beta t_k + (1-\beta) t_{k+1},$ $\beta\in[0,1]$, in the form
\begin{align}
        \newL_{d}^{\mathcal{E}}&:\quad T^* \mathcal{Q} \times T^* \mathcal{Q}\times \mathcal{N}\times \mathcal{N} \times \mathbb{R} \rightarrow \mathbb{R},\label{approxLdefU}\\
        &\;\;\;\quad \newL_{d}^{\mathcal{E}}(y_k, y_{k+1}, U_k^{(1)}, U_k^{(2)},h) =h \left[ \alpha \tilde{\mathcal{L}}_{k}^ {\mathcal{E},\gamma} + (1-\alpha) \tilde{\mathcal{L}}_{k}^ {\mathcal{E},(1-\gamma)}\right]\nonumber\\
           &\;\;\;\quad  \tilde{\mathcal{L}}_{k}^ {\mathcal{E},\gamma} = \tilde{\mathcal{L}}^{\mathcal{E}}(\overline{y}_k^\gamma, \Delta y_k, U_k^{(1)}) = \Delta \lambda_k^\top \Delta q_k + {\overline{\lambda}_k^ {\gamma}}^{T} \left(f_k^\gamma +\rho_k^ \gamma U_k^{(1)}\right)- \frac{1}{2} {U_k^{(1)}}^\top g_k^\gamma  U_k^{(1)} \nonumber\\
            &\;\;\;\quad \tilde{\mathcal{L}}_{k}^ {\mathcal{E},(1-\gamma)} = \tilde{\mathcal{L}}^{\mathcal{E}}(\overline{y}_k^{(1-\gamma)}, \Delta y_k, U_k^{(2)}) = \Delta \lambda_k^\top \Delta q_k + {\overline{\lambda}_k^{(1-\gamma)}}^{T} \left(f_k^{(1-\gamma)} +\rho_k^{(1-\gamma)} U_k^{(2)}\right)- \frac{1}{2} {U_k^{(2)}}^\top \mathrm{g}_k^{(1-\gamma)}  U_k^{(2)}, \nonumber
    \end{align}
 where we introduce the shorthands $\rho_k^\gamma = \rho(\overline{q}_k^ \gamma), \mathrm{g}_k^\gamma=\mathrm{g}(\overline{q}_k^\gamma)$.\\

\begin{remark}
    In principle, it would have been possible to use a discretisation of the controls that resembled more closely that of the state-adjoint variables $y$. Indeed, one could work with $u_k^{(1)} \approx u(t_k)$, $u_k^{(2)} \approx u(t_{k+1})$, and end up with evaluations $\bar{u}_k^{\beta} = \beta u_k^{(1)} + (1 - \beta) u_k^{(2)}$, assuming a linear approximation inside the interval. This is not recommendable for three reasons: First, one should be generally discouraged to consider $u_k^{(1)}$ and $u_k^{(2)}$ as nodal variables at all, particularly since generally $u_k^{(2)} \neq u_{k+1}^{(1)}$. Instead, in order to obtain nodal values of $u$, which are consistent among intervals, discrete mechanics tells us that one can derive them canonically from \eqref{eulerLagrangeu3} so that, for given $y_k = (q_k, \lambda_k)$,
    \begin{equation*}
        \mathrm{g}(q_k)u_k = \rho(q_k)^{\top} \lambda_k.
    \end{equation*}
    Second, it is tempting to go one step further and directly assume $u_k = u_k^{(1)},$ $u_{k+1} = u_k^{(2)}$. Using such a formulation would lead to each $u_k$ affecting contiguous intervals, resulting in smeared discrete minimisation conditions. This can lead to worst performance and spurious oscillations. Third, in some cases such as $\alpha = 0$, $\alpha = 1$ or $\gamma = \beta = 1/2$ this choice may result in an underdetermined system.
\end{remark}

The boundary velocities in terms of the discrete variables  for $\tilde{\mathcal{L}}_d^\mathcal{E}$  are given by
\begin{subequations}
\begin{align}
     v_0^{e,-}\approx v_0^{-}&=  \Delta q_0 -h\alpha\gamma\left( {f_0^\gamma}+ \rho_0^\gamma U_0^{(1)}\right) - h(1-\alpha)(1-\gamma) \left( {f_0^{(1-\gamma)}}+ \rho_0^{(1-\gamma)}U_0^{(2)}\right)\label{apVU1}\\
      v_N^{e,+}\approx v_N^{+} &=  \Delta q_{N-1} + h\alpha(1-\gamma)\left(  {f_{N-1}^\gamma}+\rho_{N-1}^{\gamma}U_{N-1}^{(1)} \right) + h(1-\alpha)\gamma\left({f_{N-1}^{(1-\gamma)}}+\rho_{N-1}^{(1-\gamma)}U_{N-1}^{(2)}\right).\label{apVU2}
\end{align}
\label{apVU_all}
\end{subequations}
This approximate form of the boundary velocities  shows explicitly that they are independent of $\lambda_d$ in the control-dependent case, as highlighted in \eqref{eq:exactdiscreteVelocitiesU}.\\

The approximate control-dependent objective function, $\tilde{\mathcal{J}}_d^\mathcal{E}$,
is obtained by inserting the approximate boundary velocity relations \eqref{apVU_all} and the approximation  of the discrete Lagrangian \eqref{approxLdefU} into \eqref{eq:exactDiscreteobjective} in place of the exact ones. 
Stationarity together with Hypothesis \ref{itm:H2p}, for simplicity's sake, lead to the following set of optimality conditions for $k = 1,\dots, N-1$,
\begin{subequations}
\begin{align}
    \delta \lambda_k &:\; \frac{1}{h^2}(q_{k+1} - 2 q_k + q_{k-1}) = \alpha \left[\gamma \left({f_k^\gamma} + \rho _k ^\gamma  U_k^{(1)}\right) + (1-\gamma)\left({f_{k-1}^\gamma} +  \rho_{k-1}^\gamma U_{k-1}^{(1)}\right)\right]\nonumber\\
        & + (1-\alpha) \left[(1-\gamma) \left({f_k^{(1-\gamma)}} +  \rho _k ^{(1-\gamma)} {U_k^{(2)}}\right) + \gamma \left({f_{k-1}^{(1-\gamma)}} + \rho_{k-1} ^{(1-\gamma)} {U_{k-1}^{(2)}}\right) \right]\label{appDEL1}\\
    \delta q_k &:\; \frac{1}{h^2}(\lambda_{k+1}-2\lambda_k + \lambda_{k-1}) =\alpha \left[ \left(\gamma D_1 f_k^\gamma - \frac{1}{h}D_2 f_k^\gamma\right)^{\!\!\top}\!\!\overline{\lambda}_k^\gamma+\left(D_{q_k} \!\!\left(\overline{\lambda}_k^\gamma \rho_k^\gamma U_k^{(1)} - \frac{1}{2}\mathrm{g}_k^\gamma(U_k^{(1)},U_k^{(1)})\right)\right)^{\!\!\top}\right] \!\!\!\nonumber\\
     & + \alpha\left[  \left((1-\gamma) D_1 f_{k-1}^\gamma + \frac{1}{h}D_2 f_{k-1}^\gamma\right)^\top\overline{\lambda}_{k-1}^\gamma+ \left(D_{q_k} \left(\overline{\lambda}_{k-1}^\gamma \rho_{k-1}^\gamma U_{k-1}^{(1)} - \frac{1}{2}\mathrm{g}_{k-1}^\gamma(U_{k-1}^{(1)},U_{k-1}^{(1)})\right)\right)^\top\right]\nonumber\\
     &  +(1-\alpha)\left[ \left((1-\gamma) D_1 f_k^{(1-\gamma)} - \frac{1}{h}D_2 f_k^{(1-\gamma)}\right)^\top\overline{\lambda}_k^{(1-\gamma)} +\left(D_{q_k}\left(\overline{\lambda}_k^{(1-\gamma)} \rho_k^{(1-\gamma)} U_k^{(2)} - \frac{1}{2}\mathrm{g}_k^{(1-\gamma)}(U_k^{(2)},U_k^{(2)})\right)\right)^\top\right]\nonumber\\
     & + (1-\alpha)\left[ \left(\gamma D_1 f_{k-1}^{(1-\gamma)} + \frac{1}{h}D_2 f_{k-1}^{(1-\gamma)}\right)^\top\overline{\lambda}_{k-1}^{(1-\gamma)}+\left(D_{q_k} \left(\overline{\lambda}_{k-1}^{(1-\gamma)} \rho_{k-1}^{(1-\gamma)} U_{k-1}^{(2)} - \frac{1}{2}\mathrm{g}_{k-1}^{(1-\gamma)}(U_{k-1}^{(2)},U_{k-1}^{(2)})\right)\right)^\top\right]\label{appDEL2}\\
    \delta U_{k}^{(1)} &:\;  {\rho_{k}^\gamma}^{T} {\overline \lambda_{k}^\gamma} -  \mathrm{g}_{k}^\gamma {U_{k}^{(1)}} = 0, \label{appDEL3} \\
    \delta U_{k}^{(2)} &:\;  {\rho_{k}^{(1-\gamma)}}^{T}{\overline \lambda_{k}^{(1-\gamma)}} -  \mathrm{g}_k^{(1-\gamma)} {U_{k}^{(2)}}=0\label{appDEL4}\\
    \delta q_0 &:\; \mu = -\Delta \lambda_0 + h\alpha\left[  \left(\gamma D_1 f_0^\gamma - \frac{1}{h}D_2 f_0^\gamma\right)^\top \overline{\lambda}_0^\gamma +\left(D_{q_0} \left(\overline{\lambda}_0^\gamma \rho_0^\gamma U_0^{(1)} - \frac{1}{2}\mathrm{g}_0^\gamma(U_0^{(1)},U_0^{(1)})\right)\right)\right]\label{apUopt1}\\
    &+h(1-\alpha)\left[  \left((1-\gamma) D_1 f_0^{(1-\gamma)} -  \frac{1}{h}D_2 f_0^{(1-\gamma)}\right)^\top\overline{\lambda}_0^{(1-\gamma)} \right.\nonumber\\
    &\left.+ \left(D_{q_0} \left(\overline{\lambda}_0^{(1-\gamma)} \rho_0^{(1-\gamma)} U_0^{(2)} - \frac{1}{2}\mathrm{g}_0^{(1-\gamma)}(U_0^{(2)},U_0^{(2)})\right)\right)^\top\right]\nonumber\\
    \delta q_{N} &:\;  D_1\phi (q_N, v_{N}^+)  = \Delta \lambda_{N-1}^\top+ h\alpha\left[ {\overline{\lambda}_{N-1}^\gamma}^\top \left((1-\gamma) D_1 f_{N-1}^\gamma +  \frac{1}{h}D_2 f_{N-1}^\gamma\right)\right.\label{apUopt2}\\
       &\left.+ D_{q_N}\left(\overline{\lambda}_{N-1}^\gamma \rho_{N-1}^\gamma U_{N-1}^{(1)} - \frac{1}{2}\mathrm{g}_{N-1}^\gamma(U_{N-1}^{(1)},U_{N-1}^{(1)})\right)\right]\nonumber\\
       & + h(1-\alpha)\left[{\overline{\lambda}_{N-1}^{(1-\gamma)}}^\top \left(\gamma D_1 f_{N-1}^{(1-\gamma)} +  \frac{1}{h}D_2 f_{N-1}^{(1-\gamma)}\right)+D_{q_N} \left(\overline{\lambda}_{N-1}^{(1-\gamma)} \rho_{N-1}^{(1-\gamma)} U_{N-1}^{(2)} - \frac{1}{2}\mathrm{g}_{N-1}^{(1-\gamma)}(U_{N-1}^{(2)},U_{N-1}^{(2)})\right)\right]\nonumber\\
       \delta \mu &:\; q_0 = q^0 ~~~~ \\
       \delta \nu &:\; v_0^{-}= \dot q ^0 \label{apUopt3}\\
        \delta v_0^- &:\;\lambda_0 = \nu\label{apUopt6}\\
        \delta v_N^+ &:\;\lambda_N^\top = - D_2 \phi (q_N, v_{N}^+),\label{apUopt7}
\end{align}
\label{apUopts}
\end{subequations}
with $g(u,u)= u^\top g u$. The fundamentally differing feature of these optimality conditions with respect to \eqref{approxnouboundaries} is the existence of the $\delta U_k^{(1)}, \delta U_k^{(2)}$ variations. Compared to the continuous setting, we can identify these as evaluations of the minimisation condition \eqref{eulerLagrangeu3}. 
Notice how when $\gamma = \beta = 1/2$,  it follows that $U_k^{(1)} = U_k^{(2)}$.\\

The rate of convergence can be derived once more from variational error analysis, as given by the following
\begin{proposition}
\label{convergencelemmau}
    Assuming sufficient differentiability, the family of low-order integration schemes defined in \eqref{approxLdefU} yield optimality conditions that approximate the exact discrete optimality conditions \eqref{exactboundaries} 
    to order 2 for $\alpha=1/2,\beta,\gamma \in [0,1]$ or $\beta,\gamma=1/2, \alpha\in [0,1]$. For all other choices $\alpha,\beta,\gamma$ the schemes approximate the exact one to order 1.
\end{proposition}
\begin{proof}
    The proof is mostly the same as in the control-independent case of Proposition \ref{convergencelemma}.
    In the exact case, besides the expansions for $y(t_k +h)$, also the control needs to be expanded $u(t_{k}+h)= u(t_k) + h\dot u(t_k) + O(h^2) = u_k + h \dot{u}_k + O(h^2)$, leading to the expansion of the Lagrangian
\begin{equation}
\tilde{\mathcal{L}}_d^{\mathcal{E},e} = \int_{t_k}^{t_{k+1}}  \tilde{\mathcal{L}}(y,\dot y,u) \,dt = h\tilde{\mathcal{L}}^\mathcal{E}_k +  \frac{h^2}{2} \left( \frac{\partial \tilde{\mathcal{L}}^\mathcal{E}_k}{\partial y}\dot {y}_k +  \frac{\tilde{\mathcal{L}}^\mathcal{E}_k}{\partial \dot{y}}\ddot{y}_k  + \frac{\tilde{\mathcal{L}}^\mathcal{E}_k}{\partial u}\dot{u}_k \right)+ O(h^3),
\label{exexpUdep}
\end{equation}
with the shorthand notation $\tilde{\mathcal{L}}^\mathcal{E}(y(t_k),\dot y(t_k), u(t_k)) = \tilde{\mathcal{L}}^\mathcal{E}_k$.
In the approximate case, the controls are evaluated at different times and have the expansions $U_k^{(1)} = u_k + h (1-\beta) \dot u_k + O(h^2), U_k^{(2)} = u_k + h \beta \dot u_k + O(h^2)$, leading to the expansion of the Lagrangian
\begin{align}
   \tilde{\mathcal{L}}_d^{\mathcal{E}} = h\tilde{\mathcal{L}}^\mathcal{E}_k +  \frac{h^2}{2} \left( 2[\alpha(1-\gamma) + (1-\alpha)\gamma]\frac{\partial \tilde{\mathcal{L}}^\mathcal{E}_k}{\partial y}\dot {y}_k +  \frac{\tilde{\mathcal{L}}^\mathcal{E}_k}{\partial \dot{y}}\ddot{y}_k  +2[\alpha(1-\beta) +(1-\alpha) \beta] \frac{\tilde{\mathcal{L}}^\mathcal{E}_k}{\partial u}\dot{u}_k \right)+ O(h^3). \label{apexpUdep}
\end{align}
Comparison of the two expansions \eqref{apexpUdep}, \eqref{exexpUdep} shows that the family of low-order integration schemes leads to second-
order approximations when 
\begin{enumerate}
    \item $\alpha=1/2$, $\beta,\gamma\in[0,1]$
    \item $\beta=\gamma=1/2,\alpha\in[0,1]$.
\end{enumerate}
Otherwise they only agree to first-order, concluding our proof. 
\end{proof}

    From Propositions \ref{convergencelemma} and \ref{convergencelemmau}  
    it follows that the new approach lets us 
    approximate $q_d,\lambda_d$ consistently, leading to methods that yield the same order of convergence for both state and costate variables $q,\lambda$.\\
\begin{remark}
\label{sameapproxremark}

    Due to the optimality conditions \eqref{appDEL3},\eqref{appDEL4} it is possible to rewrite the rest in \eqref{apUopts} to be equivalent to the control-independent ones \eqref{approxnouboundaries}. Therefore, both approaches lead to not just the same order of approximation, but are in fact the same when the same $\alpha,\gamma$ are 
    chosen. 
    Notice also that $\beta$ only has a real effect when interpreting $U_k^{(i)}$ as evaluations of the continuous control. In that regard, both from the definition of the approximation \eqref{approxLdefU} and the aforementioned optimality conditions, it only makes geometric sense for $\beta$ to coincide with $\gamma$.
\end{remark}

     Comparing the control-dependent approximation, \eqref{approxLdefU}, with the control-independent one, 
     \eqref{approxLagrangian}, shows that in the former 
     there is the additional need to discretise the controls themselves. Besides the ambiguity of choosing the discretisation, there is also an increase of variables that results in more equations 
     that need to be solved for.
$\newJ_d$ possesses $M(2(N+1)+2)$ variables $(y_d,\mu,\nu)$ to solve for, while $\newJ_d^\mathcal{E}$ depends on  $M(2(N+1)+2)+ 2S(N+1)$ variables $(y_d,U_d^{(1)},U_d^{(2)},\mu,\nu)$, where $\dim \mathcal{Q} = M, \dim \mathcal{N} = S$. This smaller number of variables means that the new approach in the control-independent case is typically faster than the dependent case due to the smaller dimensionality of the problem.

\section{Comparison with other approaches}
\label{comparison_section}
In this section, we  compare our new Lagrangian approach to other direct approaches, see e.g.~\cite{biral2016notes}, to solving the optimal control problem. In a direct method, one discretises the running cost and the dynamical constraints in \eqref{problem_definition} separately, utilizing some discretisation scheme, and then appending the discretised differential equations onto the discrete objective using Lagrange multipliers. This discrete objective function is then optimised.

A difficulty in comparing with other direct methods is, that one can arbitrarily choose discretisation schemes of one's liking, as long as the resulting equations are consistent.
However, it would be advantageous to choose schemes for the running cost and the differential equations that result in good quantitative or qualitative performance.

The new approach is compared with two different direct approaches showing how with some judicious choices they can be made equivalent.
The first direct method is based on the discretised second-order differential equations for the state variables \eqref{appDEL1}.
The second method restates the discretised state differential equation \eqref{appDEL1} in the form of a set of discrete first-order differential equations.
 
\subsection{Direct discretisation using a second-order ODE}

Let us approximate $\mathcal{J}$ \eqref{problem_definition} by performing a direct discretisation of $\hat{\mathcal{J}}^{\mathcal{E}}$ \eqref{augmentedObjective} where the controlled SODE is discretised using \eqref{appDEL1}. We discretise the running cost as
\begin{equation}
    \sum_{k=0}^{N-1} \frac{h}{2} \left(
    \alpha{U_k^{(1)} }^\top \mathrm{g}_k^\gamma U_k^{(1)} + (1-\alpha){U_k^{(2)} }^\top \mathrm{g}_k^{(1-\gamma)} U_k^{(2)}\right)\label{eq:discrete_running_cost}
\end{equation}
and append the discretised SODE using a Lagrange multiplier $\Lambda_k$. Finally, we discretise $\dot{q}(0) = v_0^-$ and $\dot{q}(T) = v_N^+$ as in \eqref{apVU_all}. This results in
\begin{align}
\tilde{\mathcal{J}}_d^\text{dir,2} &= \phi(q(T), v_N^{+}) + \mu (q_0- q^0) + \nu ( v_0^{-}- \dot q^0) + \frac{h}{2}\sum_{k=0}^{N-1} \left( \alpha{U_k^{(1)} }^\top \mathrm{g}_k^\gamma U_k^{(1)} + (1-\alpha){U_k^{(2)} }^\top \mathrm{g}_k^{(1-\gamma)} U_k^{(2)}\right) \nonumber\\
&+h\sum_{k=1}^{N-1} \Lambda_k^\top \left\{ \frac{1}{h^2}(q_{k+1} - 2 q_k + q_{k-1}) - \alpha \left[\gamma ({f_k^\gamma} +  \rho _k^\gamma {U_k^{(1)}}) + (1-\gamma) ({f_{k-1}^\gamma} +  \rho_{k-1} ^\gamma {U_{k-1}^{(1)}})\right]\right.\nonumber\\
        &\left. - (1-\alpha) \left[(1-\gamma) ({f_k^{(1-\gamma)}} +  \rho _k ^{(1-\gamma)} {U_k^{(2)}}) + \gamma ({f_{k-1}^{(1-\gamma)}} + \rho_{k-1} ^{(1-\gamma)} {U_{k-1}^{(2)}})\right] \right\}.\label{eq:Jdir2}
\end{align}

\begin{remark}
    These kinds of discretisations based on second-order ODEs appear frequently in low-order methods following the DMOC philosophy, e.g.~\cite{Schubert24}. It is rather telling that the variables involved at each node are precisely elements of $T^*\mathcal{Q}$, the configuration manifold of our new control Lagrangian, hinting at the result of the following proposition. Since these discretisations are naturally symplectic, this has implications when thought in terms of generating functions of canonical transformations \cite{goldstein02,deleon2007}. This will be discussed further in a future publication \cite{Konopik25b}.
\end{remark}

\begin{proposition}
\label{direct2eqproposition}
$\tilde{\mathcal{J}}_d^\text{dir,2}$ \eqref{eq:Jdir2} and $\tilde{\mathcal{J}}_d^{\mathcal{E}}$, as in Section \ref{sec:ApproximateOC_control_dep}, are equivalent.
\end{proposition}

\begin{proof}

    We prove equivalence using a discrete analogue of the process of integration by parts used in the continuous setting to define the new control Lagrangian in the first place. Consider the new objective function $\newJ_d^\mathcal{E}$ as defined in Section \ref{sec:ApproximateOC_control_dep}. Upon reordering of terms
     \begin{align}
        \tilde{\mathcal{J}}_d^\mathcal{E} &= \phi + \mu(q_0 - q^0) + \nu (v_0^{-} - \dot q^0) + \lambda_N v_N^{+} - \lambda_0 v_0^{-} + \frac{h}{2}\sum_{k=0}^{N-1} \left(\alpha {U_k^{(1)}}^\top \mathrm{g}_k^\gamma U_k^{(1)} +(1-\alpha) {U_k^{(2)}}^\top \mathrm{g}_k^{(1-\gamma)} U_k^{(2)}\right)\nonumber\\
        &+ h\sum_{k=1}^{N-1}\lambda_k^\top \left\lbrace \frac{1}{h^2}(q_{k+1} - 2 q_k + q_{k-1}) - \alpha \left[\gamma ({f_k^\gamma} +  \rho _k ^\gamma){U_k^{(1)}} + (1-\gamma) ({f_{k-1}^\gamma}+  \rho_{k-1} ^\gamma) {U_{k-1}^{(1)}}\right]\right.\nonumber\\
        &\left. - (1-\alpha) \left[(1-\gamma) \left({f_k^{(1-\gamma)}} +  \rho _k ^{(1-\gamma)} {U_k^{(2)}}\right) + \gamma \left({f_{k-1}^{(1-\gamma)}} +  \rho_{k-1} ^{(1-\gamma)} {U_{k-1}^{(2)}}\right)\right] \vphantom{\frac{2q}{h^2}}\right\rbrace\nonumber\\
        &+ \lambda_0^\top \left[\Delta q_0^\top -h\alpha\gamma\left( {f_0^\gamma}+ \rho_0^\gamma U_0^{(1)}\right) - h(1-\alpha)(1-\gamma) \left( {f_0^{(1-\gamma)}}+ \rho_0^{(1-\gamma)}U_0^{(2)}\right)\right]\nonumber\\
        &-\lambda_N^\top\left[\Delta q_{N-1} +h\alpha(1-\gamma)\left(  {f_{N-1}^\gamma}+\rho_{N-1}^{\gamma}U_{N-1}^{(1)} \right) + h(1-\alpha)\gamma\left({f_{N-1}^{(1-\gamma)}}+\rho_{N-1}^{(1-\gamma)}U_{N-1}^{(2)}\right)\right],\label{reorderedObj}
    \end{align}
    one can identify the objective $\newJ_d^{\text{dir},2}$ when $\Lambda_k=\lambda_k$ and realize that the last two lines in \eqref{reorderedObj} cancel the $\lambda_N^\top v_N^{+}, \lambda_0^\top v_0^{-}$ terms in the first line, concluding the proof.
\end{proof}

This results in the following trivial
\begin{corollary}
    The necessary conditions for optimality produced by $\tilde{J}_d^{\text{dir},2}$ and $\tilde{\mathcal{J}}_d^{\mathcal{E}}$ are equivalent.
\end{corollary}

\begin{remark}
    Note that after introducing the identification $\Lambda_k = \lambda_k$ from the proof of Proposition \ref{direct2eqproposition} in \eqref{eq:Jdir2}, $\lambda_0$ and $\lambda_N$ are completely absent, as shown in Proposition \ref{prp:control_dep_lambda_indep}.
\end{remark}

\subsection{Direct discretisation using a first-order ODE}
\label{firstordersection}
The prior direct approach using the controlled SODE directly is non-standard.

Usually, one works with the first-order equations, treating the position and velocity as independent variables, as discussed in Remark \ref{rmk:OCP_first_order}.
Thus, let us discretise \eqref{eq:standardAugmentedObjective} using \eqref{eq:discrete_running_cost} as running cost and the discrete first-order state equations
\begin{subequations}
\begin{align}
    \frac{1}{h}(v_{k+1}-v_k) &=   \alpha\left(f_{k}^\gamma +\rho_{k}^\gamma U_{k}^{(1)}\right) +  (1-\alpha)\left(f_{k}^{(1-\gamma)} +\rho_{k}^{(1-\gamma)} U_{k}^{(2)}\right),\label{1orderscheme1}\\
    \frac{1}{h}(q_{k+1}-q_k) &=   v_k + h\alpha\gamma\left(f_k^\gamma +\rho_k^\gamma U_k^{(1)}\right) + h(1-\alpha)(1-\gamma) \left(f_k^{(1-\gamma)} +\rho_k^{(1-\gamma)} U_k^{(2)}\right),\label{1orderscheme2}
\end{align}
\label{1orderschemeinsert}
\end{subequations}
for $k=0,\dots,N-1$.
This results in the following discrete augmented objective function:
   \begin{align}
        \mathcal{J}_d^{\text{dir},1} &= \phi(q_N, v_N ) + \mu (q_0 - q^0) + \nu (v_0 - \dot q^0) + \frac{h}{2}\sum_{k=0}^{N-1} \left( \alpha{U_k^{(1)} }^\top \mathrm{g}_k^\gamma U_k^{(1)} + (1-\alpha){U_k^{(2)} }^\top \mathrm{g}_k^{(1-\gamma)} U_k^{(2)}\right)\nonumber\\
        &+h\sum_{k=0}^{N-1} {\lambda_{k+1}^q}^\top \left[ \frac{1}{h}(q_{k+1}- q_k) - v_k - h\alpha\gamma\left(f_k^\gamma +\rho_k^\gamma U_k^{(1)}\right) - h(1-\alpha)(1-\gamma) \left(f_k^{(1-\gamma)} +\rho_k^{(1-\gamma)} U_k^{(2)}\right)\right] \nonumber\\
        &+ h\sum_{k=0}^{N-1}{\lambda_{k+1}^v}^\top \left[ \frac{1}{h}(v_{k+1} - v_{k}) -  \alpha\left(f_{k}^\gamma +\rho_{k}^\gamma U_{k}^{(1)}\right) -  (1-\alpha)\left(f_{k}^{(1-\gamma)} +\rho_{k}^{(1-\gamma)} U_{k}^{(2)}\right)\right].\label{correctfirstorderobjective}
    \end{align}

\begin{proposition}
$\tilde{\mathcal{J}}_d^\text{dir,1}$ \eqref{correctfirstorderobjective} and $\tilde{\mathcal{J}}_d^{\mathcal{E}}$, as in Section \ref{sec:ApproximateOC_control_dep}, are equivalent.
\end{proposition}

\begin{proof}
    Informed by the fact that the variation of $\mathcal{J}_d^{\text{dir},1}$ with respect to $v_k$ for $k=1,\dots, N-1$, yields
    \begin{equation}
        \delta v_k :\; \lambda_k^v - \lambda_{k+1}^v- h \lambda_{k+1}^q= 0, 
        \label{opt1dir3}
    \end{equation}
    we assume this relation is valid for $k = 0$, introducing $\lambda_0^v$, and use it to eliminate all instances of $\lambda_k^q$ in \eqref{correctfirstorderobjective}. Solving for $v_k$ in \eqref{1orderscheme2} lets us define $v_k^-(q_k,q_{k+1},U_k^{(1)},U_k^{(2)},h)$, which coincides with \eqref{apVU1} when $k = 0$. Solving for $v_{k+1}$ in \eqref{1orderscheme1}, and substituting $v_k$ by $v_k^-$ lets us define $v_{k+1}^+(q_k,q_{k+1},U_k^{(1)},U_k^{(2)},h)$, which also coincides with \eqref{apVU2} when $k = N-1$. Inserting these in place of $v_0$ and $v_N$ in the discrete cost function after having eliminated $\lambda_k^q$, and rearranging leads to
    \begin{align}
\tilde{\mathcal{J}}_d^\text{dir,1} &= \phi(q(T), v_N^+) + \mu (q_0- q^0) + \nu ( v_0^{-}- \dot q^0)+ \frac{h}{2}\sum_{k=0}^{N-1} \left( \alpha{U_k^{(1)} }^\top \mathrm{g}_k^\gamma U_k^{(1)} + (1-\alpha){U_k^{(2)} }^\top \mathrm{g}_k^{(1-\gamma)} U_k^{(2)}\right)  \nonumber\\
&+h\sum_{k=1}^{N-1} {\lambda_{k}^v}^\top \left\{ \frac{q_{k+1} - 2 q_k + q_{k-1}}{h^2} - \alpha \left[\gamma ({f_k^\gamma} +  \rho _k^\gamma {U_k^{(1)}}) + (1-\gamma) ({f_{k-1}^\gamma} +  \rho_{k-1} ^\gamma {U_{k-1}^{(1)}})\right]\right.\nonumber\\
        &\left. - (1-\alpha) \left[(1-\gamma) ({f_k^{(1-\gamma)}} +  \rho _k ^{(1-\gamma)} {U_k^{(2)}}) + \gamma ({f_{k-1}^{(1-\gamma)}} + \rho_{k-1} ^{(1-\gamma)} {U_{k-1}^{(2)}})\right]\label{dir2obj} \right\}.
\end{align}
Notice that the terms multiplied by $\lambda_0^v$ and $\lambda_N^v$ vanish identically, removing them from the discrete functional.
Identifying $\lambda_k^v = \lambda_k$ immediately yields equality of $\tilde{\mathcal{J}}_d^\text{dir,1}$ with $\tilde{\mathcal{J}}_d^\text{dir,2}$. Further, since by Proposition \ref{direct2eqproposition} the second-order direct approach is equivalent to the corresponding new Lagrangian approach, so is also the first-order approach defined in \eqref{correctfirstorderobjective} concluding the proof.
\end{proof}
Once more, the following trivial result follows.
\begin{corollary}
    The necessary conditions for optimality produced by $\tilde{J}_d^{\text{dir},1}$ and $\tilde{\mathcal{J}}_d^{\mathcal{E}}$ are equivalent.
\end{corollary}

\begin{remark}
    The first-order direct approach \eqref{1orderschemeinsert} lets us 
    identify the family of low-order integration schemes with known integration schemes for discretising the state equations:
    \begin{itemize}
        \item Midpoint rule: $\beta=\gamma=1/2,\alpha\in[0,1]$.
        \item Semi-implicit Euler $\alpha=\beta=\gamma\in \{0,1\}$ (explicit $q$, implicit $v$).
        \item Semi-implicit Euler $\gamma=\beta \in\{0,1\}, \alpha = (1-\gamma)$ (implicit $q$, explicit $v$).
        \item Partitioned trapezoidal rule:
        $\alpha=1/2,\beta=\gamma\in\{0,1\}$.
        
    \end{itemize}
\end{remark}
\begin{remark}
\label{exampleequivalence}
    The equivalence of the first-order direct approach allows us to validate the new approach optimality conditions \eqref{apUopts} with the solution of the first-order direct approach objective $\mathcal{J}_d^{\text{dir},1}$, which can be optimised via standard optimisation algorithms. For the same choice of parameters and integration scheme, the solutions have to be the same. 
\end{remark}

\section{Conserved quantities in the optimal control context -- Noether's theorem and symplecticity}
\label{symmetrysection}
Stating the optimal control problem in Lagrangian form
allows us to analyse its conserved quantities concisely via invariances of the respective Lagrangians. One such is the invariance of the optimal control problem w.r.t. a one-parameter symmetry group leading to conserved first-integrals of the Lagrangians by virtue of Noether's theorem.

Further, deriving 
variational integrators 
by approximating the exact discrete Lagrangians leads to optimality conditions that are naturally symplectic in the optimal-control state-adjoint space $T^*(T^*\mathcal{Q}).$

\subsection{Symmetries in the continuous setting}
The new Lagrangian framework 
allows us to derive conserved quantities based on the invariance of 
the corresponding Lagrangian with respect to the action of a Lie group  $\mathcal{G}$.

Following \cite{leyendecker2024new}, we consider only so-called \textit{point transformations} here. These are transformations generated by the action of $\mathcal{G}$ on the configuration space $\mathcal{Q}$ of the SODE.\\

Beginning with the control-independent case, let $\Phi : \mathcal{G} \times \mathcal{Q} \to \mathcal{Q}$ denote said action, and write $\Phi_{g}: \mathcal{Q} \to \mathcal{Q}$, with $g \in \mathcal{G}$, for the associated diffeomorphism. Since $T^*\mathcal{Q}$ is the configuration space of the new control Lagrangian, we consider the (left) cotangent lift of $\Phi$, which we denote by $\tilde{\Phi} : \mathcal{G} \times T^*\mathcal{Q} \to T^*\mathcal{Q}$ locally defined as
\begin{equation*}
    \tilde{\Phi}(g, q, \lambda) = \left(\Phi_g(q), \left(D \Phi_{g^{-1}}(q)\right)^\top \lambda \right).
\end{equation*}
Moreover, we need the tangent lift of this one, since that is the space of definition of $\tilde{\mathcal{L}}$. This is denoted as $\hat{\tilde{\Phi}} : \mathcal{G} \times T(T^*\mathcal{Q}) \to T(T^*\mathcal{Q})$, and locally defined as
\begin{align*}
    \hat{\tilde{\Phi}}(g, q, \lambda, v, v_{\lambda}) &= \left( \tilde{\Phi}_g(q, \lambda), D_1 \tilde{\Phi}_g(q, \lambda) v + D_2 \tilde{\Phi}_g(q, \lambda) v_{\lambda}
    \right)\\
    &= \left(\Phi_g(q), \left(D \Phi_{g^{-1}}(q)\right)^{\top} \lambda, D \Phi_{q}(q) v,  \left(D^2 \Phi_{g^{-1}}(q) v \right)^{\top} \lambda + \left(D \Phi_{g^{-1}}(q)\right)^\top v_{\lambda} \right).
\end{align*}
For the control-dependent case, we need to consider the action of $\mathcal{G}$ on the space of controls too. Let $\Phi^{\mathcal{E}}: \mathcal{G} \times \mathcal{E} \to \mathcal{E}$ be such an action and assume $(\mathcal{E},\pi^{\mathcal{E}},\mathcal{Q},\rho)$ is a $\mathcal{G}$-equivariant anchored vector bundle \cite{leyendecker2024new}. Then, locally
\begin{equation*}
    \Phi^{\mathcal{E}}_g(q,u) = (\Phi_g(q),\Psi_g(q) u)
\end{equation*}
with $\Psi_g: \mathcal{U} \subset \mathcal{Q} \to (\pi^{\mathcal{E}})^{-1}(\mathcal{U})$ 
providing the transformation of the control fibres and $\rho(\Phi_g(q)) \Psi_g(q) u = D \Phi_g(q) \rho(q) u$.\\

The lifted action $\hat{\tilde{\Phi}}^{\mathcal{E}} : \mathcal{G} \times T(T^*\mathcal{Q}) \oplus_{\mathcal{Q}} \mathcal{E} \to T(T^*\mathcal{Q}) \oplus_{\mathcal{Q}} \mathcal{E}$, is thus locally defined as
\begin{align*}
    \hat{\tilde{\Phi}}^{\mathcal{E}}(g, q, \lambda, v, v_{\lambda}, u) =
    &\left(\Phi_g(q), \left(D \Phi_{g^{-1}}(q)\right)^{\top} \lambda, D \Phi_{q}(q) v,  \left(D^2 \Phi_{g^{-1}}(q) v \right)^{\top} \lambda + \left(D \Phi_{g^{-1}}(q)\right)^\top v_{\lambda}, \Psi_g(q) u\right).
\end{align*}

Consider a one-parameter group of transformations, i.e. a smooth curve $g: \mathbb{R} \to \mathcal{G}$, $g_s = g(s)$, with $g_0 = e$ the identity in $\mathcal{G}$, that is also an additive subgroup.
If our control system is equivariant 
\begin{equation} \label{eq:equivariance}
    f(\Phi_{g_s}(q), D \Phi_{g_s} \dot q) + \rho (\Phi_{g_s}(q)) \Psi_{g_s}(q) u = D \Phi_{g_s} \left[ f(q,\dot q) + \rho(q) u\right]
\end{equation}
    and 
    the running cost is invariant
    \begin{equation} \label{eq:cost.invariance}
        \mathrm{g}(\Phi_{g_s}(q)) (\Psi_{g_s}(q)u, \Psi_{g_s}(q)u) = \mathrm{g}(q) (u,u)
    \end{equation}
    under the symmetry action, then the new Lagrangians $\tilde{\mathcal{L}}, \tilde{\mathcal{L}}^\mathcal{E}$ are invariant under $\mathcal{G}$ and the optimal control problem admits a conserved quantity as it was proven in \cite{leyendecker2024new}. Using Noether's theorem, it was shown that the momentum map
    \begin{align}
        I(y,\dot y) &= \frac{\partial \newL}{\partial \dot y}(y,\dot y) \cdot \left.\frac{d}{ds}\right|_{s=0}\tilde{\Phi}_{g_s}(y)  = \frac{\partial \newL^\mathcal{E}}{\partial \dot y}(y,\dot y,u) \cdot \left.\frac{d}{ds}\right|_{s=0}\tilde{\Phi}_{g_s}(y) \equiv p^{\top} \cdot \left.\frac{d}{ds}\right|_{s=0}\tilde{\Phi}_{g_s}(y)
        \label{contconserved}
    \end{align}
    is a first integral of the corresponding flow, for optimal solution curves $(y,u)$.

    \subsection{Symmetries in the discrete setting}

    The group action $\tilde{\Phi}_{g}$ 
    induces an action on the discrete product space \cite{marsden2001discrete},  here $T^*\mathcal Q\times T^*\mathcal Q$, canonically, called the product or diagonal action 
    \begin{align}
     \Phi_{g}^{T^*\mathcal Q\times T^*\mathcal Q}(y_k,y_{k+1}) &= (\tilde{\Phi}_{g}(y_k),\tilde{\Phi}_{g}(y_{k+1})).\label{liftedaction}
     \end{align}
      We say that a 
     discrete Lagrangian $\newL_d$ is invariant under the diagonal action $\Phi_{g}^{T^*\mathcal Q\times T^*\mathcal Q}$ 
     if 
    \begin{align}
       \left(\newL_d \circ \Phi_{g}^{T^*\mathcal Q\times T^*\mathcal Q}\right)(y_k, y_{k+1})  &=  \newL_d(y_k, y_{k+1}).
        \label{discreteInvariance}
    \end{align}
    If $\tilde{\mathcal{L}}$ is invariant with respect to $\hat{\tilde{\Phi}}_{g}$, it is immediate to show from its definition that $\tilde{\mathcal{L}}_d^e$ will be invariant with respect to $\Phi_{g}^{T^*\mathcal Q\times T^*\mathcal Q}$. However, if $\tilde{\mathcal{L}}_d$ is an approximation of it, this invariance may or may not be inherited.\\
    
    A discrete version of Noether's theorem was proven in \cite{marsden2001discrete}, Theorem 1.3.3. This result is readily applicable in our case for $\newL_d: T^*\mathcal{Q} \times T^* \mathcal{Q} \to \mathbb{R}$ as discrete Lagrangian, and $\Phi_{g}^{T^*\mathcal Q\times T^*\mathcal Q}$ as group action. It says that if $\newL_d$ is invariant under the action of a one-parameter group of transformations $g_s$
    then,
    
    \begin{equation}
        I_d = (p_{y,k}^-)^\top \left.\frac{d}{ds}\right|_{s=0} \tilde{\Phi}_{g_s}(y_k)= (p_{y,k}^+)^\top \left.\frac{d}{ds}\right|_{s=0} \tilde{\Phi}_{g_s}(y_k),
        \label{eq:Noether_integral}
    \end{equation}
    with $p_{y,k}^-=p_{y,k}^+$ the momentum of $y_k$ defined in \eqref{eq:discreteMomentadef}, is conserved. This discrete first integral is equivalent to the continuous one in \eqref{contconserved} evaluated at the nodes.\\

     We now turn our attention to the control-dependent case, where it is necessary to separate between exact and approximate cases.
      In the semi-discrete case $\newL_d^{\mathcal{E},e}$, the group action $\tilde{\Phi}_{g}$ 
    induces an action on $T^*\mathcal Q\times T^*\mathcal Q\times C^1([0,T],\mathcal{E})$ (see footnote in Definition \ref{def:newControlDepDiscLag}), of the form
        \begin{equation}
            \Phi_{g}^{T^*\mathcal{Q}\times T^*\mathcal{Q}\times \mathcal{E}} (y_k,y_{k+1},u_{k,k+1}) = (\tilde{\Phi}_{g}(y_k),\tilde{\Phi}_{g}(y_{k+1}),\tilde{u}_{k,k+1})
                \label{controlledLiftedAction}
        \end{equation}  
            with
      \begin{equation*}
            \tilde{u}_{k,k+1}(t) = \Psi_{g}(\pi^{\mathcal{E}}(u_{k,k+1}(t)))\, u_{k,k+1}(t), \quad t \in [t_k,t_{k+1}].
      \end{equation*}
      The control-dependent Lagrangian $\tilde{\mathcal{L}}_d^{\mathcal{E},e}$ is said to be invariant under the 
      action $\Phi_{g}^{T^*\mathcal{Q}\times T^*\mathcal{Q}\times \mathcal{E}}$ 
      if the following holds
    \begin{equation}
        \left(\tilde{\mathcal{L}}_d^{\mathcal{E},e} \circ \Phi_{g}^{T^*\mathcal{Q}\times T^*\mathcal{Q}\times \mathcal{E}}  \right) (y_k,y_{k+1},u_{k,k+1}) = \tilde{\mathcal{L}}_d^{\mathcal{E},e} (y_k, y_{k+1}, u_{k,k+1}).
    \end{equation}

    This definition of the symmetry action allows us to state the following 
    \begin{theorem}[Exact semi-discrete Noether's theorem]
    \label{discretefirstIntegral}
             If $\tilde{\mathcal{L}}_d^{\mathcal{E},e}$ is invariant with respect to the action of a one-parameter group of transformations ${g_s}\in \mathcal{G}$, then the first integral \eqref{eq:Noether_integral} with $p_{y,k}^-=p_{y,k}^+$ the momentum of $y_k$ defined in Proposition \ref{prp:UDiscreteFibreDer} is conserved along solutions $(y_d,u_d)$ of \eqref{exactboundaries}.
    \end{theorem}
    \begin{proof}
        From the invariance of the Lagrangian $\newL_d^{\mathcal{E},e}$ under the 
        symmetry action \eqref{controlledLiftedAction}, it follows that
        \begin{equation*}
            \left.\frac{d}{ds} \newL_d^{\mathcal{E},e}(\tilde{\Phi}_{g_s}(y_k),\tilde{\Phi}_{g_s}(y_{k+1}),\tilde{u}_{k,k+1})  \right|_{s=0} =\left.\frac{d}{ds} \newL_d^{\mathcal{E},e}(y_k,y_{k+1},u_{k,k+1})  \right|_{s=0} = 0\,.
        \end{equation*}
        Furthermore,
        \begin{align*}
            &\left.\frac{d}{ds} \newL_d^{\mathcal{E},e}(\tilde{\Phi}_{g_s}(y_k),\tilde{\Phi}_{g_s}(y_{k+1}),\tilde{u}_{k,k+1})  \right|_{s=0} = D_1 \newL_{d,k}^{\mathcal{E},e} \left.\frac{d}{ds} \tilde{\Phi}_{g_s}(y_k)\right|_{s=0} + D_2 \newL_{d,k}^{\mathcal{E},e} \left.\frac{d}{ds} \tilde{\Phi}_{g_s}(y_{k+1})\right|_{s=0} \\
            &\hspace{3.1cm}+ \int_{0}^{h} D_3\newL^{\mathcal{E}}(y_{k,k+1},\dot{y}_{k,k+1}, u_{k,k+1}(t)) \frac{d}{ds}\Psi_{g_s}(\pi^\mathcal{E}(u_{k,k+1}(t)))  u_{k,k+1}(t) \left.\vphantom{\frac{d}{d}}\right\vert_{s=0} \!\!dt\\
            &\hspace{3.1cm}= -(p_k^{-} )^\top \left.\frac{d}{ds} \tilde{\Phi}_{g_s}(y_k)\right|_{s=0} + (p_{k+1}^{+})^\top \left.\frac{d}{ds} \tilde{\Phi}_{g_s}(y_{k+1})\right|_{s=0}\,,
        \end{align*}
        where we have made use of the fact that $D_3 \newL^\mathcal{E}$ is just the optimality condition \eqref{eulerLagrangeu3} and thus solution curves $y_d,u_d$ lead to the vanishing of these contributions. From the last line the result of the theorem follows. 
    \end{proof}
    
    As in the control-independent case, it is immediate to show that if $\tilde{\mathcal{L}}^\mathcal{E}$ is invariant, then so is $\tilde{\mathcal{L}}_d^{\mathcal{E},e}$.

    When the exact semi-discrete Lagrangian is approximated, the functional dependence on $u_{k,k+1}$ is replaced by a set of discrete values. In our case, the symmetry action for the family of low-order integration schemes is defined by
    \begin{subequations}
    \begin{align}
        \tilde{U}_k^{(1)} &=  \Psi_{g_s}(\bar{q}_{k}^{\beta})\, U_k^{(1)}\\
        \tilde{U}_k^{(2)} &=  \Psi_{g_s}(\bar{q}_{k}^{(1-\beta)})\, U_k^{(2)},
    \end{align}
    \label{USymmetryactions}
    \end{subequations}
    where $\bar{q}_{k}^{\beta} = \pi_{\mathcal{Q}}(\bar{y}_{k}^{\beta})$.  This defines the fully discrete control-dependent symmetry action $\Phi_{g_s}^{T^*\mathcal{Q}\times T^*\mathcal{Q} \times \mathcal{N}\times \mathcal{N} }$.\\

    We say that the approximate Lagrangian is invariant under this symmetry action if
    \begin{subequations}
    \begin{align}
        \left(\newL_d^\mathcal{E}\circ \Phi_{g_s}^{T^*\mathcal{Q}\times T^*\mathcal{Q} \times \mathcal{N}\times \mathcal{N}}\right)(y_k, y_{k+1},U_k^{(1)},U_k^{(2)}) &= \newL_d^\mathcal{E}(y_k, y_{k+1},U_k^{(1)},U_k^{(2)}).
    \end{align}
    \label{approxInvariance}
    \end{subequations}
      \begin{theorem}
    \label{approxnoetherpropo}
        If $\newL_d^\mathcal{E}$ is invariant with respect to the action $ \Phi_{g_s}^{T^*\mathcal{Q}\times T^*\mathcal{Q} \times \mathcal{N}\times \mathcal{N} }$ of a one-parameter group of transformations $g_s \in \mathcal{G}$, then the first integral \eqref{eq:Noether_integral},
        where $p_{y,k}^{+}, p_{y,k}^{-}$ are the momenta of the corresponding discrete Lagrangian $\newL_d^\mathcal{E}$,
        is conserved along
        solutions $(y_d, U_d^{(1)},U_d^{(2)})$. 
    \end{theorem}
    \begin{proof}
        Variation via the symmetry action in the invariant case directly leads to
        \begin{align*}
            0&=\left.\frac{d}{ds}\newL_d^\mathcal{E}(\tilde{\Phi}_{g_s}(y_k), \tilde{\Phi}_{g_s}(y_{k+1}),\tilde{U}_k^{(1)},\tilde{U}_k^{(2)})\right|_{s=0} \\
            &=D_1 \newL_{d,k}^\mathcal{E}  \left.\frac{d}{ds}\tilde{\Phi}_{g_s}(y_k)\right|_{s=0} + D_2 \newL_{d,k}^\mathcal{E}  \left.\frac{d}{ds}\tilde{\Phi}_{g_s}(y_{k+1})\right|_{s=0} + D_3 \newL_{d,k}^\mathcal{E}  \left.\frac{d}{ds}\tilde{U}_k^{(1)}\right|_{s=0} +  D_4 \newL_{d,k}^\mathcal{E}  \left.\frac{d}{ds}\tilde{U}_k^{(2)}\right|_{s=0}\\
            &=-(p_k^{-})^\top \left.\frac{d}{ds}\tilde{\Phi}_{g_s}(y_k)\right|_{s=0} + (p_{k+1}^{+})^\top \left.\frac{d}{ds}\tilde{\Phi}_{g_s}(y_{k+1})\right|_{s=0} ,
        \end{align*}
        where going from the second to the third line we used the fact that for optimal  $(y_d,U_d^{(1)},U_d^{(2)})$ the minimisation conditions give  $D_3 \newL_{d,k}^{\mathcal{E}}=D_4 \newL_{d,k}^{\mathcal{E}}=0$,
        from which the result follows directly. 
    \end{proof}

    Furthermore, the invariance of the continuous Lagrangians $\newL,\newL^\mathcal{E}$ subject to linear symmetry actions is preserved by the approximate Lagrangians $\newL_d,\newL^\mathcal{E}_d$  using the family of low-order integration schemes considered in this work as is shown in the following

    \begin{theorem}
        Let $(\mathcal{E},\pi^{\mathcal{E}},\mathcal{Q},\rho)$ be a $\mathcal{G}_1$-equivariant anchored vector bundle with $\mathcal{Q} = \mathbb{R}^n$ and $\mathcal{G}_1$ 
        a one-parameter group of affine transformations acting on $\mathcal{Q}$ by $\Phi_{g_s}(q) = A_s q + c_s$, with $A_s \in \mathbb{R}^{n \times n}, c_s \in \mathbb{R}^{n}$, $s \in \mathbb{R}$. Further, in the control-dependent case, assume that either
        \begin{enumerate}
            \item $\Psi_{g_s}(q) = E_s$, with $E_s \in \mathbb{R}^{m\times m}$,  or
            \item $\beta = \gamma$.
        \end{enumerate}
        Then, if the continuous Lagrangians $\newL,\newL^\mathcal{E}$ are invariant under the 
        corresponding lifted symmetry actions, then the approximate discrete Lagrangians $\newL_d,\newL_d^{\mathcal{E}}$ 
        obtained via the family of low-order integration schemes \eqref{approxLagrangian}, \eqref{approxLdefU} will preserve this invariance. Consequently, the first integral
        \begin{equation}
            I_d = p_q^{\top} (B q + d) - p_{\lambda}^{\top} B^{\top} \lambda, \quad \text{with } B = \left.\frac{\partial A_{s}}{\partial s}\right|_{s=0}, d = \left.\frac{\partial c_{s}}{\partial s}\right|_{s=0}
            \label{firstInt_linearaction}
        \end{equation}
        is conserved along the corresponding approximate discrete flows.
        \label{theoSym}
    \end{theorem}
    \begin{proof}
    We focus on the control-dependent case, since the proof for the control-independent case follows similarly, without the complications due to the controls. The lifted action and its discrete counterpart are    
    \begin{align*}    \hat{\tilde{\Phi}}_{g_s}^{\mathcal{E}}(q, \lambda, v_q, v_{\lambda}, u) &=
    \left(A_s q + c_s, A_s^{-\top} \lambda, A_s v_{q}, A_s^{-\top} v_{\lambda}, \Psi_g(q) u\right).\\
        \Phi_{g_s}^{T^*\mathcal{Q} \times T^*\mathcal{Q}\times \mathcal{N}\times\mathcal{N}}(q_k,\lambda_k,q_{k+1},\lambda_{k+1},U_k^{(1)}, U_k^{(2)}) &= (A_s q_k+c_s, A_s^{-\top} \lambda_k,A_s q_{k+1}+c_s, A_s^{-\top} \lambda_{k+1},\tilde{U}_k^{(1)}, \tilde{U}_k^{(2)})
    \end{align*}
    where $A^{-\top} = (A^{-1})^{\top} = (A^{\top})^{-1}$ and 
    $\tilde{U_k}^{(1)}, \tilde{U_k}^{(2)}$ defined via \eqref{USymmetryactions}
    which leads to
    \begin{align}
        &(\newL_{d,k}^\mathcal{E}\circ\Phi_{g_s}^{T^*\mathcal{Q} \times T^*\mathcal{Q}\times \mathcal{N}\times\mathcal{N}})(y_k,y_{k+1},U_k^{(1)},U_k^{(2)}) = \alpha \newL^\mathcal{E}(A_s \overline{q}_k^\gamma+ c_s, A_s^{-\top}\overline{\lambda}_k^\gamma, A_s \Delta q_k, A_s^{-\top}\Delta \lambda_k, \tilde{U}_k^{(1)} ) \\
        &\hspace{5.3cm}+ (1-\alpha) \newL^\mathcal{E}(A_s \overline{q}_k^{(1-\gamma)}+c_s, A_s^{-\top}\overline{\lambda}_k^{(1-\gamma)}, A_s \Delta q_k, A_s^{-\top}\Delta \lambda_k,\tilde{U}_k^{(2)} ).\nonumber
    \end{align}
    Under the assumption of the theorem, the Lagrangian $\tilde{\mathcal{L}}^\mathcal{E} $ is invariant under the action $\hat{\tilde{\Phi}}_{g_s}^{\mathcal{E}}$,
    which holds true for evaluations at arbitrary $(q,\lambda,v_q,v_{\lambda},u)$ and thus also for the evaluations of $\newL^\mathcal{E}$ at the points $( \overline{q}_k^\gamma, \overline{\lambda}_k^\gamma,  \Delta q_k, \Delta \lambda_k, U_k^{(1)} )$ and $(\overline{q}_k^{(1-\gamma)}, \overline{\lambda}_k^{(1-\gamma)},  \Delta q_k, \Delta \lambda_k,{U}_k^{(2)} )$, when $\beta=\gamma$ or when $\Psi_{g_s}(q)$ is independent of $q$, i.e. $\Psi_{g_s}(q) = E_s$. 
    This results in the invariance of $\newL_{d,k}^\mathcal{E}$ under $\Phi_{g_s}^{T^*\mathcal{Q} \times T^*\mathcal{Q}\times \mathcal{N}\times\mathcal{N}}$. 
    Following Theorem \ref{approxnoetherpropo}, invariance of $\newL_d^\mathcal{E}$ leads to conservation of a first integral which in 
     this case is exactly \eqref{firstInt_linearaction} due to the fact that
     \begin{equation*}
         \left. \frac{d}{ds} \tilde{\Phi}_{g_s}(y)\right\vert_{s=0}= \left( \left. \frac{d}{ds} (A_s q + c_s) \right\vert_{s=0}, \left. \frac{d}{ds} A_s^{-\top} \lambda\right\vert_{s=0}\right)
     \end{equation*}
     together with the identity
     \begin{equation*}
         \left. \frac{d}{ds} A_s^{-1} \right\vert_{s=0} = -\left. A_s^{-1} \left( \frac{d}{ds} A_s \right) A_s^{-1} \right\vert_{s=0}\,.\qedhere
     \end{equation*}
    
    \end{proof}

    \begin{remark}
        Here we observe once more, as previously mentioned in Remark \ref{sameapproxremark}, that $\beta = \gamma$ is the natural geometric choice since when performing evaluations of the Lagrangian, it is expected that $(q,u) \in \mathcal{E}$. Any different choice of $\beta$ implies that $u$ is a vector whose base is $q' \neq q$, leading to improper transformations unless $\Psi_{g_s}$ is independent of the point.
    \end{remark}
    
 The integration scheme considered in this work is the family of low-order integration schemes \eqref{approxLagrangian}, \eqref{approxLdefU} for $\newL_d^e,\newL_d^{\mathcal{E},e}$.
However, by virtue of $\newL$ being an actual Lagrangian on $T^*\mathcal{Q}$
, it is possible to infer a number of general symmmetry properties  for a variety of integration schemes. 
\begin{remark}
First, Runge-Kutta-based approximations of 
discrete Lagrangians lead automatically to variationally partitioned Runge-Kutta methods \cite{marsden2001discrete} in 
$T^* (T^*\mathcal{Q})$, which preserve quadratic first integrals of the form $Q(y,p_y) = p_y D y$  \cite{marsden2001discrete,suris1990,hairergeometric} where $D$ is a 
matrix of appropriate dimensions. In the case of linear symmetry actions of the form $\left.\frac{d}{ds}\right|_{s=0} \tilde{\Phi}_{g_s}(y_k) = B y_k$, e.g. rotational symmetries, the conserved first integral is then of the form $I_d=p_k B y_k$, a special form of the conserved first integrals \eqref{firstInt_linearaction}.

Second, given $\newL$ invariant under the lifted action $\hat{\tilde{\Phi}}_g$, a Galerkin-based approximation 
whose interpolation polynomials are equivariant w.r.t.~this action lead to 
invariant discrete Lagrangians 
\cite{ober2015construction} and the first integral \eqref{eq:Noether_integral} will be consequently preserved. Since 
any interpolation polynomial is equivariant w.r.t. affine symmetry actions, it follows \cite{ober2015construction} that approximating the exact discrete Lagrangian $\newL_d^e$ via Galerkin methods will preserve 
\eqref{firstInt_linearaction} too.
\end{remark}

    \subsection{Symplecticity in the optimal control space}
    An advantage of deriving optimality conditions variationally from the discrete Lagrangians is that, by construction, they result in symplectic methods, regardless of the integration scheme chosen \cite{marsden2001discrete}. Consequently, optimality conditions \eqref{approxnouboundaries} of $\newL_d$ form 
    naturally symplectic discretisations 
    on $T^*(T^*\mathcal{Q})$ of the continuous optimality conditions for $\newL$. Symplectic methods result in approximate numerical solutions that have good qualitative behavior. First, the corresponding symplectic form is conserved which leads to generally bounded energy errors. In the context of new control Lagrangians $\tilde{\mathcal{L}},\tilde{\mathcal{L}}^\mathcal{E}$ this means that
    Pontryagin's control Hamiltonian
    \begin{equation}
        \label{eq:PonytryaginsHamiltonian}
        \mathcal{H}(q,v,\lambda_q,\lambda_v,u^*) = \lambda_q^{\top} v + \lambda_v^{\top} \left( f(q,\dot{q}) + \rho(q) \,u^* \right) - \frac{1}{2} \, {u^*}^{\top}\mathrm{g}(q)\, u^*,
    \end{equation}
    where $(\lambda_q,\lambda_v)$ were introduced in Remark \ref{rmk:OCP_first_order}, as well as the new control Hamiltonian \cite{leyendecker2024new}
    \begin{subequations}
    \begin{align}
        \tilde{\mathcal{H}}(y,p_y) = \tilde{\mathcal{H}}^\mathcal{E}(y,p_y,u^*) &= \frac{\partial \tilde{\mathcal{L}}^\mathcal{E}}{\partial \dot y} p_y - \tilde{\mathcal{L}}^\mathcal{E}(y,\dot y,u^*) = p_q^{\top} p_\lambda -\lambda^\top (f(q,p_\lambda) + \rho(q) u^*) + \frac{1}{2}\mathrm{g}(q)(u^*,u^*)
        \end{align}
        \label{hamiltonDef}
    \end{subequations}
    for $u^*$ satisfying \eqref{eulerLagrangeu3}
    are approximately conserved, i.e. while they are not exactly constant, their error is bounded.
    
    Second, these lead to typically very good long-time behaviour of the solutions which may be useful for optimal control problems 
    over long periods of time, in particular in astronomy and molecular dynamics \cite{chyba2009role}.

    Third, Offen and McLachlan showed \cite{mclachlan2019symplectic} that symplectic methods applied to Hamiltonian boundary value problems preserve certain types of bifurcations that non-symplectic methods cannot. While the particular relevance of this fact for the optimal control problem is not yet understood, it still hints at the suitability of these methods.

    It is important to stress once more that here symplecticity of the discrete flow is meant with respect to the natural symplectic structure of the state-adjoint space $T^*(T^*\mathcal{Q})$ at the optimum, not necessarily of the state space of the forced mechanical system. The properties of the restriction of the resulting schemes to the state space will be explored
    in a future publication.


\section{Example -- Low thrust orbital transfer}
\label{examplesection}
In this section we apply the necessary optimality conditions derived for the family of low-order integration schemes \eqref{approxnouboundaries} in order to solve a specific example: the low-thrust orbital transfer of a satellite. The satellite orbits a planet from an orbit of radius $r$ to another orbit of radius $r'$ in the same plane, allowing a two-dimensional description of the state-space and thus $r = \sqrt{x(0)^2+y(0)^2}$,  $r' = \sqrt{x(T)^2+y(T)^2}$.

 \begin{figure}[h]
     \centering
\includegraphics[width=0.951\textwidth]
{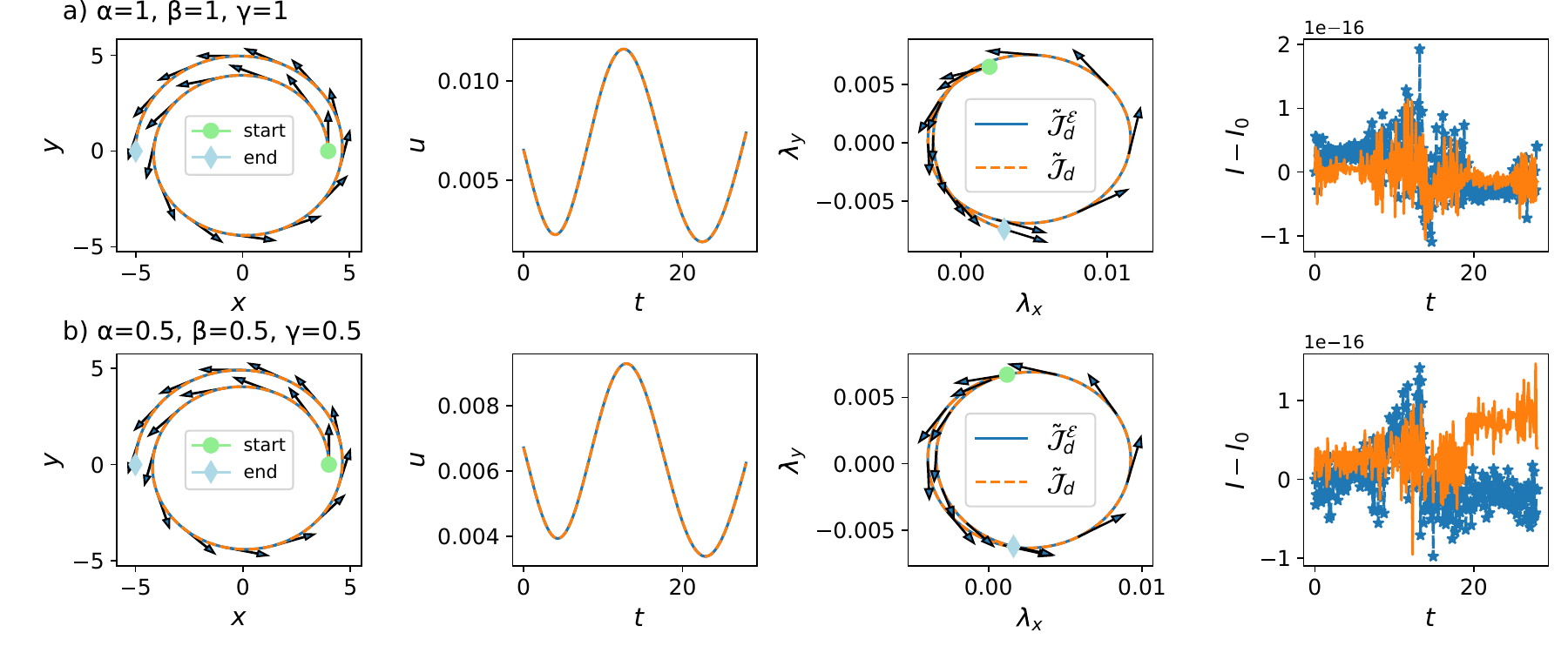}
\caption{Example solutions of a low-thrust orbital transfer from $q(0) = (4,0)$ to $q(T) = (-5,0)$ for $T = 28$ with $d=1.5$, $M=10, g = 1, h=0.1$, $K_q = K_v = \mathrm{Id}_{2 \times 2}$ with the parameters for the a) row being  $\alpha=1=\beta=\gamma$, bottom row b) $\alpha=0.5=\beta=\gamma$. Shown for each choice of the integration scheme are the control-dependent case $\tilde{\mathcal{L}}_d^\mathcal{E}$ (blue lines)  and the independent case $\tilde{\mathcal{L}}_d$ (orange lines). The rows contain from left to right: ($x,y$)-trajectory with example state-velocity vectors $v$, control $u$ evolution, ($\lambda_x,\lambda_y$)-trajectory with example costate-velocity vectors $v_\lambda$, error evolution of the conserved quantity $I-I(0)$.}
\label{fig1}
\end{figure}
The mechanical system is defined by the following Lagrangian in Cartesian coordinates
    \begin{equation}
        L(q,\dot q) = L(x,y,\dot x, \dot y) = \frac{1}{2}m(\dot x^2 + \dot y^2) + G \frac{Mm}{r},
        \label{orbitL}
    \end{equation}
    with the mass of the satellite $m$, the mass of the planet $M \gg m$, and the gravitational constant $G$. The force is assumed to act in the azimuthal direction, i.e.~$f_L = mru \mathrm{d} \varphi$. Therefore, the force-controlled Euler-Lagrange equations in Cartesian coordinates may be written as 
    \begin{align}
        \ddot q = 
        \begin{pmatrix}
            \ddot x\\
            \ddot y
        \end{pmatrix}=
        -\frac{GM}{r^3}\begin{pmatrix}
            x\\
            y
        \end{pmatrix} 
        + \frac{1}{r}\begin{pmatrix}
            -y\\
            x
        \end{pmatrix} u = f(q) + \rho(q) u,    \label{orbit2ODE}
    \end{align}
    the control space metric is obtained as $\mathrm{g}(q) = \rho(q)^T \mathrm{g}_\mathcal{Q}(q) \rho(q) =1$ where $\mathrm{g}_\mathcal{Q}(q) = \mathrm{Id}_{2 \times 2}$ is the Euclidean metric in $\mathbb{R}^2$. We assume a fixed time-horizon $T = d \sqrt{\frac{4\pi^2}{8\gamma M}(r^0 + r^T)^3}$, with $d$ the number of revolutions. The optimal control problem with initial position and velocity constraints may then be stated generically in the form
    \begin{equation}
        \begin{aligned}
        J(u) &= \phi(q(T),\dot q(T)) + \int_0^T \frac{1}{2} u(t)^T \mathrm{g}(q) u(t) dt\\
        \text{subject to:}&\\
        q(0) &=q^0= (x^0,0),\\
        \dot q(0) &=\dot{q}^0= (0,\sqrt{\gamma M/x^0})\\
         \ddot q &= f(q) + \rho(q) u,\\
        \end{aligned}
        \label{orbitOptiProb}
    \end{equation}

\begin{figure}[ht]
\centering
\includegraphics[width=0.6\textwidth]
{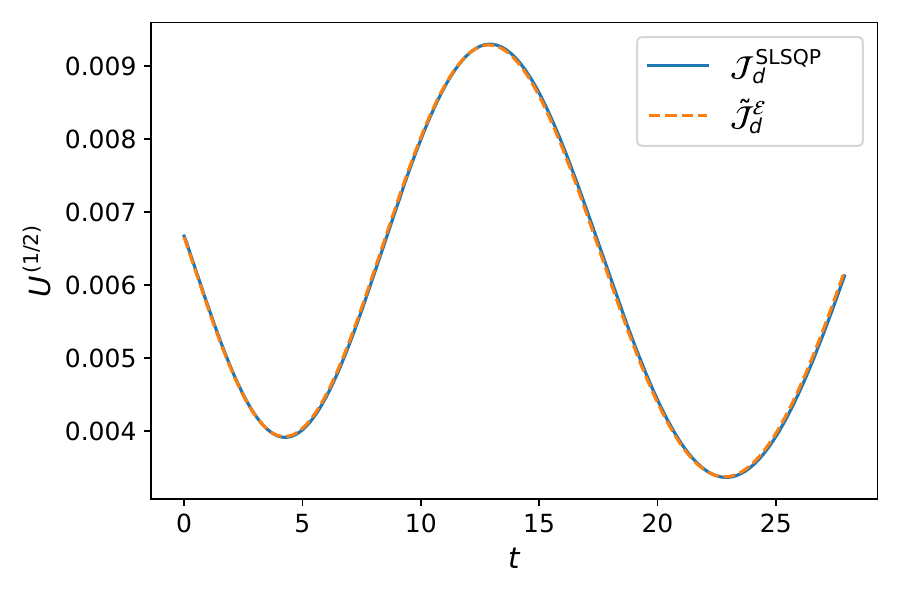}
\caption{Comparison of solving the optimal control  problem via the new approach \eqref{apUopts} and a standard solution algorithm.
 The parameters are the same as in Figure \ref{fig1}, for step size $h=0.1$ and midpoint evaluations $\alpha=\beta=\gamma=1/2$. The new approach yields the same optimal solution.}
\label{fig2}
\end{figure}
\begin{figure}[ht]
\centering
\includegraphics[width=1\textwidth]
{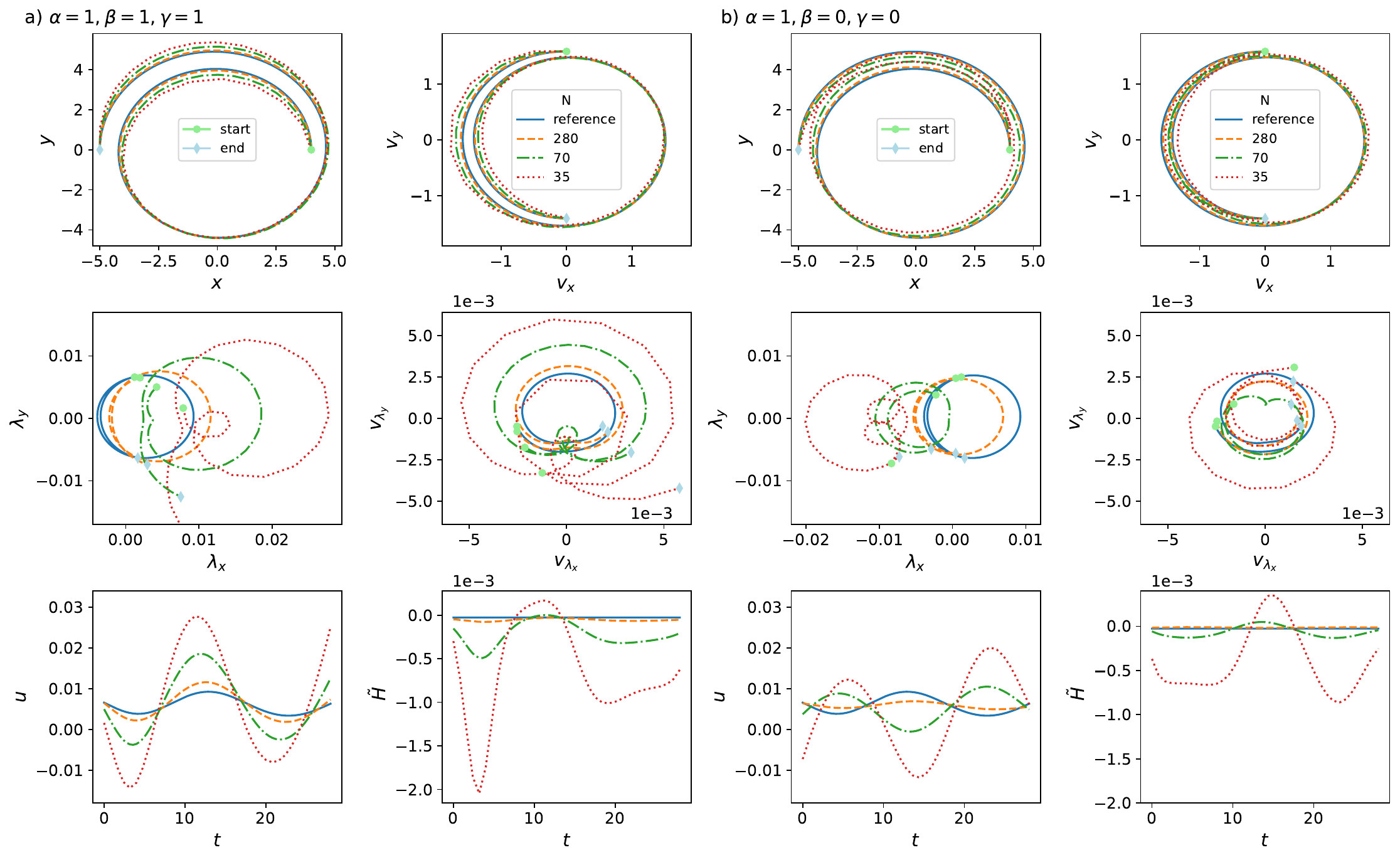}
\caption{Convergence of the method as a function of the number of steps of the discretisation, $N$. Here the first-order methods with a) $ \alpha=\beta=\gamma=1$ (left two columns) and b) $ \alpha=1,~ \beta=\gamma=0$ (right two columns) are considered with otherwise the same parameters as in Figure \ref{fig1}. The figures for a) and b) are the same, being in the first row the $(x,y)$-trajectory and  $(v_x,v_y)$-trajectory, in the second row the $(\lambda_x,\lambda_y)$-trajectory and $(v_{\lambda_x},v_{\lambda_y})$-trajectory. The third row contains the control $u$ evolution and the evolution of the control Hamiltonian  $\tilde{\mathcal{H}}$. The number of steps correspond to step-sizes of $h=[0.1,0.4,0.8]$. The reference trajectories correspond to the solution of the problem for $h=0.01$ with the second-order method $\alpha=\beta=\gamma = 0.5$. Both methods approach the reference solution for decreasing $h$ in all aspects.}
\label{fig3}
\end{figure}
     with terminal cost $\phi$ steering the system towards the final condition $(q^T, \dot{q}^T)$ 

    \begin{equation}
        \phi(q(T),\dot q(T)) = \left(q(T)-q^T\right)^\top K_q \left(q(T)-q^T\right) + \left(\dot q(T) -\dot q^T\right)^\top K_v \left(\dot q(T) -\dot q^T\right).
    \end{equation}
    $K_q$ and $K_v$ are positive-definite matrices acting as weights.
     This optimal control problem, following \cite{leyendecker2024new},  possesses the conserved quantity
 \begin{equation}
     I(q,\lambda,\dot q, \dot \lambda) = -\dot x \lambda_y + \dot y \lambda_x + x \dot{\lambda}_y - y \dot \lambda_x,
     \label{thrust_conserved}
 \end{equation}
    which is associated to the invariance with respect to planar rotations.
    
    Using the continuous Legendre transform \eqref{contLegendre}, it is possible to rewrite \eqref{thrust_conserved} in the form
    \begin{equation}
        I(q,\lambda,p_q,p_\lambda) = \lambda_x p_{\lambda_y} - \lambda_y p_{\lambda_x}   + x p_y  - y p_x .\label{thrust_cons_reformulated}
    \end{equation}

    The optimal control problem \eqref{orbitOptiProb} may now be solved by finding the solutions of the optimality conditions derived for the Lagrangians $ \tilde{\mathcal{L}}_d^\mathcal{E},\tilde{\mathcal{L}}_d$ with the low-order integration scheme \eqref{approxLdefU} leading to \eqref{apUopts} for $ \tilde{\mathcal{L}}_d^\mathcal{E}$ and \eqref{approxLagrangian}, leading to \eqref{approxnouboundaries}  for $ \tilde{\mathcal{L}}_d$, via e.g. root finding algorithms. 

    Figure \ref{fig1} shows example solutions for a choice of parameters of the  problem for two different integration scheme choices, a) semi-implicit Euler, b) midpoint and both the control-dependent $\newL_d^\mathcal{E}$ and independent $\newL_d$ descriptions. Note that the step size $h = 0.1$ is rather large, particularly for the Kepler system and the semi-implicit Euler, which is of order 1. However, the solution is already quite close to that of the midpoint rule, which is of order 2. The two cases show very similar behaviour in various features like its state and costate trajectories and corresponding velocities $v,v_\lambda$, calculated canonically via $v_k^ {\pm}, v_{\lambda,k}^ {\pm}$, optimal controls canonically calculated from the minimisation condition \eqref{eulerLagrangeu3}. As expected, the low-order family of integration schemes conserves the first integral  $I$ in \eqref{thrust_cons_reformulated},  corresponding to the rotational symmetry of the optimal control problem.

    In Section \ref{firstordersection} it is shown that our new Lagrangian method is equivalent to a first-order direct approach. Following Remark \ref{exampleequivalence}, this equivalence allows us to validate the solution of our new method with a more standard calculation method.
    The running cost \eqref{eq:discrete_running_cost} with  $\alpha=\beta=\gamma=1/2$ together with the same terminal cost $\phi$ as in \eqref{correctfirstorderobjective}, can be optimised as a function of the midpoint evaluated control $U^{(1)}_k =U^{(2)}_k= U^{(1/2)}_k$, subject to the initial position and velocity constraints as well as the midpoint discretised first-order dynamical constraints \eqref{1orderschemeinsert}. This minimisation is accomplished by using one of the standard minimisation solvers of the python SciPy library, \texttt{scipy.optimize.minimize} with the method \texttt{'SLSQP'}, solving for the optimal control curve $U^{(1/2)}_k$. Figure \ref{fig2} shows that the new approach indeed yields the same solution as this method, denoted as $\mathcal{J}_d^{\mathrm{SLSQP}}$, within tolerances.

In Figure \ref{fig3}, solutions of the new method for a) $\alpha=\beta=\gamma=1$, b) $\alpha=1,\beta=\gamma=0$, both different variants of semi-implicit Euler, are shown for various number of steps and compared to a finer solution computed with the midpoint rule as reference. Even for the chosen first-order method, the solutions are still relatively close to the others for the small number of steps $N=35$. Interestingly, the state trajectory differs less than the costate trajectories in dependence of the step-size. This may be partially explained by the fact that the costate equations have completely free boundaries, while the state equations have prescribed boundary values. 
The scaling of the controls with $N$ in Figure \ref{fig3} is probably due to the nature of the dynamics of the problem itself. In the case of Figure \ref{fig3}.a), since the position of the satellite initially undershoots as $N$ increases while the velocity overshoots, stronger controls need to be exerted to compensate for this, leading to another overshoot later on and further compensation. In the case of Figure \ref{fig3}.b) the situation is almost the reverse of a), which leads to the inversion of the controls for low $N$. The new control Hamiltonian $\tilde{\mathcal{H}}$, \eqref{hamiltonDef}, also shows the expected behaviour of not being preserved exactly. This holds true also for the Pontryagin's control Hamiltonian $\mathcal{H}$, \eqref{eq:PonytryaginsHamiltonian}, as they are equivalent here, up to a sign difference.

\section{Conclusion and future work}

 %
 In this work, the discretisation of the new Lagrangian approach to optimal control problems with quadratic cost functions and affine-control systems proposed in \cite{leyendecker2024new} was developed and analysed.

 In the process, we introduced an exact discrete and a novel exact semi-discrete version of the problem based on the theory of discrete mechanics. This approach provides a theoretical framework that lets us analyse some general features of the discrete problem and paves the way to approximations conducive to numerical methods.
 
 While we were able to show that the exact discrete setting leads indeed to equivalent necessary optimality conditions to the continuous ones, the resulting boundary terms 
 revealed interesting features. In particular, in the control-dependent case the initial and final costates, $\lambda_0$ and $\lambda_N$, turn out to be completely absent from the cost function 
 (see Remark \ref{rmk:discrete_velocity_lambda}). 
 However, in the control-independent formulation, this is not the case and it may be possible to recuperate their values naturally from its variations.
 
 We then proposed 
 a family of low-order integration schemes, which were used to approximate the corresponding exact Lagrangians.
 This family depends on 
 parameters $\alpha,\beta,\gamma$, and we were able to show that it leads to the same optimality conditions for the control-dependent and independent case for the same choice of parameters. Further, it was 
 possible to show that 
 equivalent direct methods can be constructed using discretisations of the first and second-order dynamic equations. Their construction, however, requires some forethought and judicious choices.

 Using the discrete version of Noether's theorem, we were able to derive conserved quantities of the exact discrete and semi-discrete optimal control problem. Moreover, using these results we were able to prove general conservation properties of the proposed low-order family.

 Lastly, the low-thrust orbital transfer was considered as a numerical example to show the validity of the derived optimality conditions as well as the preservation properties derived earlier.

In future works, we intend to further develop the exact discrete and semi-discrete versions of an OCP as showcased in this work. The ideas are fairly general and can be applied directly to formulate a discrete counterpart to the more general version of this new Lagrangian formulation in \cite{Konopik25b}. We also aim to obtain rigorous conditions under which the semi-discrete version exists. Also, while the new approach leads to variational integrators that are symplectic in the state-adjoint space $T^*(T^*\mathcal{Q})$, the resulting integrators for the underlying 
state equations, as well as for the adjoint equations, are not necessarily (forced) symplectic discretisations. The necessary modifications to the new Lagrangian theory and its consequences to its geometry will be 
present in \cite{Konopik25b} while numerical studies will be presented in its discrete follow-up.\\
In Remark \ref{rmk:on_equivalence_hypotheses} we mentioned that we believe that small-time local controllability was key to show the that Hypothesis \ref{itm:hyp1} of Theorem \ref{thm:equivalence_no_u} holds. Further analysis will be required to solidify this into a theorem.
Of further interest are the implications of the results of Offen and McLachlan \cite{mclachlan2019symplectic} on symplectic integration of boundary value problems in the optimal control context.
Lastly, 
we would like to extend the formulation to more general integration schemes like partitioned Runge-Kutta or Galerkin methods with which one can derive high-order integrators.

\section*{Acknowledgements}
The authors acknowledge the support of Deutsche Forschungsgemeinschaft (DFG) with the projects: LE 1841/12-1, AOBJ: 692092 and OB 368/5-1, AOBJ: 692093.
\section*{authors' contributions}
RS: investigation, methodology, conceptualization, formal analysis, writing - original draft, writing - review \& editing\\
MK: investigation, software, visualisation, formal
 analysis, writing - original draft, writing - review \& editing\\
SM: methodology, conceptualization, writing - review and editing\\
SL, SOB: conceptualisation, supervision, writing - review \& editing, funding acquisition
\section*{Data availability}
The graphs can be plotted and their data for all the graphs generated using the code \cite{Code}.

\begingroup\scriptsize

\makeatletter
\renewcommand\@openbib@code{\itemsep\z@}
\makeatother
\printbibliography
\endgroup
\end{document}